\documentclass[12pt,reqno]{amsart}

\usepackage{hyperref}
\usepackage[usenames]{color}
\usepackage[latin1]{inputenc} 
\usepackage{float}
\usepackage{bbm}

\usepackage[english]{babel} 

\usepackage{amssymb, amsmath}
\usepackage{geometry,graphicx}

\newcommand{\Cov}{\mathbb{C}\mathrm{ov}}
\newcommand{\e}{\mathrm{e}}
\newcommand{\dif}{\mathrm{d}}
\newcommand{\leb}{\mathrm{L}}

\theoremstyle{plain}
\newtheorem{theorem}{Theorem}[section] 

\newtheorem{lemma}[theorem]{Lemma} 

\newtheorem{remark}[theorem]{Remark}
\newtheorem{notation}[theorem]{Notation}
\newtheorem{example}[theorem]{Example}

\numberwithin{equation}{section} 
\title[\hfill\protect\parbox{0.95\linewidth}{~~On the approximation of the probability density function of the \\randomized non-autonomous complete linear differential equation}]{On the approximation of the probability density function of the randomized non-autonomous complete linear differential equation}

\author{J. Calatayud, J.-C. Cort\'{e}s, M. Jornet} 

\begin{document}

\maketitle

\begin{center}
\noindent
\address{Instituto Universitario de Matem\'{a}tica Multidisciplinar,\\
Universitat Polit\`{e}cnica de Val\`{e}ncia,\\
Camino de Vera s/n, 46022, Valencia, Spain\\
email: jucagre@alumni.uv.es; jcortes@imm.upv.es; marcjor@alumni.uv.es
}
\end{center}

\begin{abstract}
In this paper we study the randomized non-autonomous complete linear differential equation. The diffusion coefficient and the source term in the differential equation are assumed to be stochastic processes and the initial condition is treated as a random variable on an underlying complete probability space. The solution to this random differential equation is a stochastic process. Any stochastic process is determined by its finite-dimensional joint distributions. In this paper, the main goal is to obtain the probability density function of the solution process (the first finite-dimensional distribution) under mild conditions. The solution process is expressed by means of Lebesgue integrals of the data stochastic processes, which, in general, cannot be computed in an exact manner, therefore approximations for its probability density function are constructed. The key tools applied to construct the approximations are the Random Variable Transformation technique and Karhunen-Loève expansions. Our results can be applied to a large variety of examples. Finally, several numerical experiments illustrate the potentiality of our findings. 

\end{abstract}

\section{Introduction and motivation}

The behaviour of physical phenomena is often governed by chance and it does not follow strict deterministic laws. Specific examples can be found in many areas, for example,  in Thermodynamics, the analysis of crystal lattice dynamics in which there is a small percentage of small atoms, randomness arises since atom masses may take two or more values according to probabilistic laws \cite{Ebeling_termodinamica}; in analysis of free vibrations in Mechanics, the heterogeneity in material may be very complicated, so it may be more suitable to  model it via appropriate random fields \cite{Wirsching_vibraciones}; in Epidemiology, the rate of transmission  of a disease depends upon very complex factors, such as genetic, weather, geography, etc., that can be better described using adequate probabilistic distributions rather than deterministic tools, etc.  \cite{Linda_Allen}. As a consequence, it is reasonable to deal with mathematical models considering randomness in their formulation. Motivated by this fact, i.e. the random nature involved in numerous physical phenomena together with ubiquity of deterministic  differential equations to formulate classical laws in Thermodynamics, Mechanics, Epidemiology, etc., it is natural to  randomize  classical differential equations  to describe  mathematical models. This approach leads to the areas of Stochastic Differential Equations (SDEs)  and  Random Differential Equations (RDEs). In the former case, differential equations are forced by stochastic processes having an irregular sample behaviour (e.g., nowhere differentiable) such as the Wiener process or Brownian motion. SDEs are written as Itô or Stratonovich integrals rather than in their differential form. This approach permits dealing with uncertainty via very irregular stochastic processes like white noise, but assuming specific probabilistic properties (Gaussianity, independent increments, stationarity, etc.). The rigorous treatment of SDEs requires a special stochastic calculus, usually referred to as Itô calculus, whose cornerstone result is Itô's Lemma \cite{Oksendal, Kloeden_Platen}. While RDEs are those in which random effects are directly manifested in their input data (initial/boundary conditions, source term and coefficients) \cite{Soong, llibre_smith}. An important advantage of RDEs is that a wider range of probabilistic patterns are allowed for input data like Beta, Gamma, Gaussian distributions (including Brownian motion), but not white noise. Furthermore, the analysis of RDEs takes advantage of classical calculus where powerful tools are available \cite{Soong}. When dealing with both SDEs and RDEs, the main goals are to compute, exact or numerically, the solution stochastic process, say $x(t)$, and its main statistical functions (mostly mean, $\mathbb{E}[x(t)]$, and variance, $\mathbb{V}[x(t)]$).  A more ambitious target is to compute its $n$-dimensional probability density distribution, $f_n(x_1,t_1;\ldots;x_n,t_n)$, whenever it exists, which gives the probability density distribution of the random vector $(x(t_1),\ldots,x(t_n))$, i.e., the joint distribution of the solution at $n$ arbitrary time instants $t_i$, $1\leq i \leq n$ \cite[p.34-36]{Soong}. This is generally a very difficult goal, and in practice significant efforts have been made to compute just the first probability density function, $f_1(x,t)$, since from it all one-dimensional statistical moments of the stochastic process $x(t)$ can be derived using the fact that
\[
\mathbb{E}[(x(t))^k]
=
\int_{-\infty}^{\infty}
x^k f_1(x,t)\, \dif x,\qquad k=1,2,\ldots.
\]
This permits computing the variance, skewness, kurtosis, etc., as well as determining the probability that the solution lies in a specific interval of interest
\[
\mathbb{P}[a\leq x(t) \leq b]
=
\int_{a}^{b}
 f_1(x,t)\, \dif x ,
\]
for every time instant $t$.

In the context of SDEs, it is known  that  $f_n(x_1,t_1;\ldots;x_n,t_n)$ satisfies a Fokker-Planck partial differential equation that needs to be solved, exact or numerically  \cite{Oksendal, Kloeden_Platen}, while Random Variable Transformation technique \cite[ch.~6]{Soong}, \cite[ch.~5 and ch.~6]{Papoulis_Pillai} stands out as a powerful strategy to address this problem in the framework of RDEs. Here, we restrict  our analysis to an important class of RDEs taking advantage of the Random Variable Transformation technique. This method  has been successfully applied to study significant random ordinary differential equations \cite{Dorini_CNSNS_2016,JC_CNSNS_2014}  and random partial differential equations \cite{Dorini_JCP,Selim_AMC_2011,Selim_EPJP_2015,Selim_2013_JQSRT,Xu_AMM_2016}, that appear in different areas such as Epidemiology, Physics, Engineering, etc.  In all these  interesting contributions, uncertainty is considered via random variables rather than stochastic processes. From a mathematical standpoint, this fact restricts the generality of previous contributions.   

This paper is devoted to computing approximations of the first probability density function to the random non-autonomous complete linear differential equation under very general hypotheses. As we shall see later, our approach takes advantage of the aforementioned Random Variable Transformation technique together with Karhunen-Loève expansion \cite[ch.~5]{llibre_powell}. It is important to point out that  this fundamental problem has not been addressed in its general formulation yet. Indeed, to the best of our knowledge, the computation of the first probability density function of the solution of random first and second-order linear differential equations  has been recently tackled in \cite{AAA_RODE_Lineal_1} and \cite{Mediterranean_Lineal_2}, respectively, but only in the case that coefficients are random variables rather than stochastic processes. Recently, the results \cite{AAA_RODE_Lineal_1} have been extended to random first-order linear systems in \cite{SRODE_Romanian} but assuming that coefficients  are random variables.   

For the sake of clarity, first we introduce some notations and results that will be required throughout the paper. We will also establish some results based upon the sample approach to the random non-autonomous non-homogeneous linear differential equation that are aimed to complete our analysis. 

Consider the non-autonomous complete linear differential equation 
\begin{equation}
 \begin{cases} x'(t)=a(t)x(t)+b(t),\;t\in [t_0,T], \\ x(t_0)=x_0, \end{cases} 
\label{edo_det}
\end{equation}
where $a(t)$ is the diffusion coefficient, $b(t)$ is the source term and $x_0$ is the initial condition. The formal solution to this Cauchy problem is given by
\begin{equation}
 x(t)=x_0\,\e^{\int_{t_0}^t a(s)\,\dif s}+\int_{t_0}^t b(s)\,\e^{\int_s^t a(r)\,\dif r}\,\dif s, 
\label{sol_det}
\end{equation}
where the integrals are understood in the Lebesgue sense.

Now we consider (\ref{edo_det}) in a random setting, meaning that we are going to work on an underlying complete probability space $(\Omega,\mathcal{F},\mathbb{P})$, where $\Omega$ is the set of outcomes, that will be generically denoted by $\omega$, $\mathcal{F}$ is a $\sigma$-algebra of events and $\mathbb{P}$ is a probability measure. We shall assume that the initial condition $x_0(\omega)$ is a random variable and 
\[ a=\{a(t,\omega):\,t_0\leq t\leq T,\,\omega\in\Omega\},\quad b=\{b(t,\omega):\,t_0\leq t\leq T,\,\omega\in\Omega\} \] are stochastic processes defined in $(\Omega,\mathcal{F},\mathbb{P})$. In this way, the formal solution to the randomized initial value problem (\ref{edo_det}) is given by the following stochastic process:
\begin{equation}
 x(t,\omega)=x_0(\omega)\,\e^{\int_{t_0}^t a(s,\omega)\,\dif s}+\int_{t_0}^t b(s,\omega)\,\e^{\int_s^t a(r,\omega)\,\dif r}\,\dif s, 
\label{sol_ale}
\end{equation}
where the integrals are understood in the Lebesgue sense. 

\begin{notation}
Throughout this paper we will work with Lebesgue spaces. Remember that, if $(S,\mathcal{A},\mu)$ is a measure space, we denote by $\leb^p(S)$ or $\leb^p(S,\dif\mu)$ ($1\leq p<\infty$) the set of measurable functions $f:S\rightarrow\mathbb{R}$ such that $\|f\|_{\leb^p(S)}=(\int_S |f|^p\,\dif \mu)^{1/p}<\infty$. We denote by $\leb^\infty(S)$ or $\leb^\infty(S,\dif\mu)$ the set of measurable functions such that $\|f\|_{\leb^\infty(S)}=\inf\{\sup\{|f(x)|:\,x\in S\backslash N\}:\,\mu(N)=0\}<\infty$. We write a.e. as a short notation for ``almost every'', which means that some property holds except for a set of measure zero.

Here, we will deal with several cases: $S=\mathcal{T}\subseteq\mathbb{R}$ and $\dif\mu=\dif x$ the Lebesgue measure, $S=\Omega$ and $\mu=\mathbb{P}$ the probability measure, and $S=\mathcal{T}\times\Omega$ and $\dif\mu=\dif x\times \dif\mathbb{P}$. Notice that $f\in \leb^p(\mathcal{T}\times\Omega)$ if and only if $\|f\|_{\leb^p(\mathcal{T}\times \Omega)}=(\mathbb{E}[\int_{\mathcal{T}} |f(x)|^p\,\dif x])^{1/p}<\infty$. In the particular case of $S=\Omega$ and $\mu=\mathbb{P}$, the short notation a.s. stands for ``almost surely''.

In this paper, an inequality related to Lebesgue spaces will be frequently used. This inequality is well-known as the generalized Hölder's inequality, which says that, for any measurable functions $f_1,\ldots,f_m$,
\begin{equation}
 \|f_1\cdots f_m\|_{\leb^1(S)}\leq \|f_1\|_{\leb^{r_1}(S)}\cdots \|f_m\|_{\leb^{r_m}(S)}, 
 \label{Holder}
\end{equation}
where
\begin{equation}
 \frac{1}{r_1}+\cdots+\frac{1}{r_m}=1,\quad 1\leq r_1,\ldots,r_m\leq\infty. 
 \label{coefHolder}
\end{equation}
When $m=2$, inequality (\ref{Holder})-(\ref{coefHolder}) is simply known as Hölder's inequality. When $m=2$, $r_1=2$ and $r_2=2$, inequality (\ref{Holder})-(\ref{coefHolder}) is termed Cauchy-Schwarz inequality.
\end{notation}

For the sake of completeness, we establish under which hypotheses on the data stochastic processes $a$ and $b$ and in what sense the stochastic process given in (\ref{sol_ale}) is a rigorous solution to the randomized problem (\ref{edo_det}).

\begin{remark}
To better understand the computations in the proof of Theorem \ref{absc}, let us recall some results that relate differentiation and Lebesgue integration. Recall that a function $f:[T_1,T_2]\rightarrow\mathbb{R}$ belongs to $\mathrm{A.C}([T_1,T_2])$ ($\mathrm{A.C}$ stands for absolutely continuous) if there exists its derivative $f'$ at a.e. $x\in [T_1,T_2]$, $f'\in \leb^1([T_1,T_2])$ and $f(x)=f(T_1)+\int_{T_1}^x f'(t)\,\dif t$ for all $x\in [T_1,T_2]$ (i.e., $f$ satisfies the Fundamental Theorem of Calculus for Lebesgue integration). Equivalently, $f\in\mathrm{A.C}([T_1,T_2])$ if for all $\epsilon>0$ there exists a $\delta>0$ such that, if $\{(x_k,y_k)\}_{k=1}^m$ is any finite collection of disjoint open intervals in $[T_1,T_2]$ with $\sum_{k=1}^m (y_k-x_k)<\delta$, then $\sum_{k=1}^m |f(y_k)-f(x_k)|<\epsilon$. Equivalently, $f\in\mathrm{A.C}([T_1,T_2])$ if there exists $g\in\leb^1([T_1,T_2])$ such that $f(x)=f(T_1)+\int_{T_1}^x g(t)\,\dif t$ for all $x\in [T_1,T_2]$. In such a case, $g=f'$ almost everywhere on $[T_1,T_2]$. For more on these statements, see \cite[p.129]{lahiri} \cite[Th.2, p.67]{weir}. 

Two results concerning absolute continuity will be used:
\begin{itemize}
\item[i)] If $f\in\mathrm{A.C}([T_1,T_2])$, then $\e^f\in \mathrm{A.C}([T_1,T_2])$ (use the Mean Value Theorem to $\e^f$ and the $\epsilon$-$\delta$ definition for absolute continuity).
\item[ii)] The product of absolutely continuous functions is absolutely continuous.
\end{itemize}
\end{remark}

\begin{theorem} \label{absc}
If the data processes $a$ and $b$ of the randomized initial value problem (\ref{edo_det}) satisfy $a(\cdot,\omega),b(\cdot,\omega)\in \leb^1([t_0,T])$ for a.e. $\omega\in\Omega$, then the solution process $x(t,\omega)$ given in (\ref{sol_ale}) has absolutely continuous sample paths on $[t_0,T]$, $x(t_0,\omega)=x_0(\omega)$ a.s. and, for a.e. $\omega\in\Omega$, $x'(t,\omega)=a(t,\omega)x(t,\omega)+b(t,\omega)$ for a.e. $t\in [t_0,T]$. Moreover, this is the unique absolutely continuous process being a solution.

On the other hand, if $a$ and $b$ have continuous sample paths on $[t_0,T]$, then $x(t,\omega)$ has $C^1([t_0,T])$ sample paths and $x'(t,\omega)=a(t,\omega)x(t,\omega)+b(t,\omega)$ for all $t\in [t_0,T]$ and a.e. $\omega\in\Omega$. Moreover, this is the unique differentiable process being a solution.
\end{theorem}
\begin{proof}
Rewrite (\ref{sol_ale}) as
\[ x(t,\omega)=\e^{\int_{t_0}^t a(s,\omega)\,\dif s}\left\{x_0(\omega)+\int_{t_0}^t b(s,\omega)\e^{-\int_{t_0}^s a(r,\omega)\,\dif r}\,\dif s\right\}. \]

Since $a(\cdot,\omega)\in \leb^1([t_0,T])$, the function $\int_{t_0}^t a(s,\omega)\,\dif s$ belongs to $\mathrm{A.C}([t_0,T])$, for a.e. $\omega\in\Omega$. Therefore, by i), $\e^{\int_{t_0}^t a(s,\omega)\,\dif s}$ belongs to $\mathrm{A.C}([t_0,T])$, for a.e. $\omega\in\Omega$, with derivative $a(t,\omega)\e^{\int_{t_0}^t a(s,\omega)\,\dif s}$. 

On the other hand,
\begin{align*}
|b(s,\omega)|\e^{-\int_{t_0}^s a(r,\omega)\,\dif r}\leq {} & |b(s,\omega)|\e^{\left|\int_{t_0}^s a(r,\omega)\,\dif r\right|}\leq |b(s,\omega)|\e^{\int_{t_0}^s |a(r,\omega)|\,\dif r} \\
\leq {} & |b(s,\omega)|\e^{\int_{t_0}^T |a(r,\omega)|\,\dif r}=|b(s,\omega)|\e^{\|a(\cdot,\omega)\|_{\leb^1([t_0,T])}}\in \leb^1([t_0,T],\dif s),
\end{align*}
for a.e. $\omega\in\Omega$. Then the function $\int_{t_0}^t b(s,\omega)\e^{-\int_{t_0}^s a(r,\omega)\,\dif r}\,\dif s$ belongs to $\mathrm{A.C}([t_0,T])$, for a.e. $\omega\in\Omega$, with derivative $b(t,\omega)\e^{-\int_{t_0}^t a(r,\omega)\,\dif r}$. 

As the product of absolutely continuous functions is absolutely continuous (see ii)), we derive that, for a.e. $\omega\in\Omega$, $x(\cdot,\omega)\in \mathrm{A.C}([t_0,T])$. Moreover, the product rule for the derivative yields, for a.e. $\omega\in\Omega$, $x'(t,\omega)=a(t,\omega)x(t,\omega)+b(t,\omega)$ for a.e. $t\in [t_0,T]$.

For the uniqueness, we apply Carathéodory's Existence Theorem \cite[p.30]{carath}.

If $a$ and $b$ have continuous sample paths on $[t_0,T]$, one has to use the Fundamental Theorem of Calculus for the Riemann integral, instead. The uniqueness comes from the global version of the Picard-Lindelöf Theorem, or, if you prefer, by standard results on the deterministic linear differential equation.
\end{proof}

\section{Obtaining the probability density function of the solution}

The main goal of this paper is, under suitable hypotheses, to compute approximations of the probability density function, $f_1(x,t)$, of the solution stochastic process given in (\ref{sol_ale}), $x(t,\omega)$, for $t\in [t_0,T]$. To achieve this goal, we will use the Karhunen-Loève expansions for both data stochastic processes $a$ and $b$. 

Hereinafter, the operators $\mathbb{E}[\cdot]$, $\mathbb{V}[\cdot]$ and $\Cov[\cdot,\cdot]$ will denote the expectation, the variance and the covariance, respectively. 

We state two crucial results that will be applied throughout our subsequent analysis.

\begin{lemma}[Random Variable Transformation technique] \label{lema_abscont}
Let $X$ be an absolutely continuous random vector with density $f_X$ and with support $D_X$ contained in an open set $D\subseteq \mathbb{R}^n$. Let $g:D\rightarrow\mathbb{R}^n$ be a $C^1(D)$ function, injective on $D$ such that $Jg(x)\neq0$ for all $x\in D$ ($J$ stands for Jacobian). Let $h=g^{-1}:g(D)\rightarrow\mathbb{R}^n$. Let $Y=g(X)$ be a random vector. Then $Y$ is absolutely continuous with density
\begin{equation}
 f_Y(y)=\begin{cases} f_X(h(y))|Jh(y)|,&\;y\in g(D), \\ 0,&\; y\notin g(D). \end{cases}  
\label{dens_lema}
\end{equation}
\end{lemma}

The proof appears in Lemma 4.12 of \cite{llibre_powell}.

\begin{lemma}[Karhunen-Loève Theorem] \label{KLlemma}
Consider a stochastic process $\{X(t):\,t\in\mathcal{T}\}$ in $\leb^2(\mathcal{T}\times\Omega)$. Then
\begin{equation}
 X(t,\omega)=\mu(t)+\sum_{j=1}^\infty \sqrt{\nu_j}\,\phi_j(t)\xi_j(\omega), 
 \label{karh}
\end{equation}
where the sum converges in $\leb^2(\mathcal{T}\times\Omega)$, $\mu(t)=\mathbb{E}[X(t)]$, $\{\phi_j\}_{j=1}^\infty$ is an orthonormal basis of $\leb^2(\mathcal{T})$, $\{(\nu_j,\phi_j)\}_{j=1}^\infty$ is the set of pairs of (nonnegative) eigenvalues and eigenvectors of the operator
\begin{equation}
 \mathcal{C}:\leb^2(\mathcal{T})\rightarrow \leb^2(\mathcal{T}),\quad \mathcal{C}f(t)=\int_{\mathcal{T}} \Cov[X(t),X(s)]f(s)\,\dif s, 
 \label{karhC}
\end{equation}
and $\{\xi_j\}_{j=1}^\infty$ is a sequence of random variables with zero expectation, unit variance and pairwise uncorrelated. Moreover, if $\{X(t):\,t\in\mathcal{T}\}$ is a Gaussian process, then $\{\xi_j\}_{j=1}^\infty$ are independent and Gaussian.
\end{lemma}

The proof appears in Theorem 5.28 of \cite{llibre_powell}.

\begin{remark}
When the operator $\mathcal{C}$ defined in (\ref{karhC}) has only a finite number of nonzero eigenvalues, then the process $X$ of Lemma \ref{KLlemma} can be expressed as a finite sum:
\[ X(t,\omega)=\mu(t)+\sum_{j=1}^J \sqrt{\nu_j}\,\phi_j(t)\xi_j(\omega). \]
In the subsequent development, we will write the data stochastic processes $a$ and $b$ via their Karhunen-Loève expansions. The summation symbol in the expansion will be always written up to $\infty$ (the most difficult case), although it could be possible that their corresponding covariance integral operators $\mathcal{C}$ have only a finite number of nonzero eigenvalues. In such a case, when we write vectors of the form $(\xi_1,\ldots,\xi_N)$ for $N\geq1$ later on (for instance, see the hypotheses of the forthcoming theorems or the approximating densities (\ref{f1n}), (\ref{f1Nhomo}), etc.), we will interpret that we stop at $N=J$ if $J<\infty$. 
\end{remark}

\bigskip

Suppose that $a,b\in \leb^2([t_0,T]\times\Omega)$. Then, according to Lemma \ref{KLlemma}, we can write their Karhunen-Loève expansion as
\[ a(t,\omega)=\mu_a(t)+\sum_{j=1}^\infty \sqrt{\nu_j}\,\phi_j(t)\xi_j(\omega), \quad b(t,\omega)=\mu_b(t)+\sum_{i=1}^\infty \sqrt{\gamma_i}\,\psi_i(t)\eta_i(\omega), \]
where $\{(\nu_j,\phi_j)\}_{j=1}^\infty$ and $\{(\gamma_i,\psi_i)\}_{i=1}^\infty$ are the corresponding pairs of (nonnegative) eigenvalues and eigenfunctions, $\{\xi_j\}_{j=1}^\infty$ are random variables with zero expectation, unit variance and pairwise uncorrelated, and $\{\eta_i\}_{i=1}^\infty$ are also random variables with zero expectation, unit variance and pairwise uncorrelated. We will assume that both sequences of pairs $\{(\nu_j,\phi_j)\}_{j=1}^\infty$ and $\{(\gamma_i,\psi_i)\}_{i=1}^\infty$ do not have a particular ordering. In practice, the ordering will be chosen so that the hypotheses of the theorems stated later on are satisfied (for example, if we say in a theorem that $\xi_1$ has to satisfy a certain condition, then we can reorder the pairs of eigenvalues and eigenvectors and the random variables $\xi_1,\xi_2,\ldots$ so that $\xi_1$ satisfies the condition).

We truncate the Karhunen-Loève expansions up to an index $N$ and $M$, respectively:
\[ a_N(t,\omega)=\mu_a(t)+\sum_{j=1}^N \sqrt{\nu_j}\,\phi_j(t)\xi_j(\omega), \quad b_M(t,\omega)=\mu_b(t)+\sum_{i=1}^M \sqrt{\gamma_i}\,\psi_i(t)\eta_i(\omega). \]
This allows us to have a truncation for the solution given in (\ref{sol_ale}):
\begin{align*}
 x_{N,M}(t,\omega)= {} & x_0(\omega)\,\e^{\int_{t_0}^t a_N(s,\omega)\,\dif s}+\int_{t_0}^t b_M(s,\omega)\,\e^{\int_s^t a_N(r,\omega)\,\dif r}\,\dif s \\
= {} & x_0(\omega)\,\e^{\int_{t_0}^t \left(\mu_a(s)+\sum_{j=1}^N \sqrt{\nu_j}\,\phi_j(s)\xi_j(\omega)\right)\,\dif s} \\
+ {} & \int_{t_0}^t \left(\mu_b(s)+\sum_{i=1}^M \sqrt{\gamma_i}\,\psi_i(s)\eta_i(\omega)\right)\,\e^{\int_s^t \left(\mu_a(r)+\sum_{j=1}^N \sqrt{\nu_j}\,\phi_j(r)\xi_j(\omega)\right)\,\dif r}\,\dif s.
\end{align*}

For convenience of notation, we will denote (in bold letters) $\pmb{\xi}_N=(\xi_1,\ldots,\xi_N)$ and $\pmb{\eta}_M=(\eta_1,\ldots,\eta_M)$, understanding this as a random vector or as a deterministic real vector, depending on the context. We also denote
\[ K_a(t,\pmb{\xi}_N)=\int_{t_0}^t \left(\mu_a(s)+\sum_{j=1}^N \sqrt{\nu_j}\,\phi_j(s)\xi_j\right)\,\dif s, \]
\[ S_b(s,\pmb{\eta}_M)=\mu_b(s)+\sum_{i=1}^M \sqrt{\gamma_i}\,\psi_i(s)\eta_i. \]
In this way,
\begin{align}
 x_{N,M}(t,\omega)= {} & x_0(\omega)\e^{K_a(t,\pmb{\xi}_N(\omega))}+\int_{t_0}^t S_b(s,\pmb{\eta}_M(\omega))\e^{K_a(t,\pmb{\xi}_N(\omega))-K_a(s,\pmb{\xi}_N(\omega))}\,\dif s  \nonumber \\
= {} & \e^{K_a(t,\pmb{\xi}_N(\omega))}\left\{x_0(\omega)+\int_{t_0}^t S_b(s,\pmb{\eta}_M(\omega))\e^{-K_a(s,\pmb{\xi}_N(\omega))}\,\dif s\right\}. \label{xNM}
\end{align}

We assume that $x_0$ and $(\xi_1,\ldots,\xi_N,\eta_1,\ldots,\eta_M)$ are absolutely continuous and independent, for all $N,M\geq1$. The densities of $x_0$ and $(\xi_1,\ldots,\xi_N,\eta_1,\ldots,\eta_M)$ will be denoted by $f_0$ and $f_{(\xi_1,\ldots,\xi_N,\eta_1,\ldots,\eta_M)}$, respectively.

Under this scenario described, in the following subsections we will analyze how to approximate the probability density function of the solution stochastic process $x(t,\omega)$ given in (\ref{sol_ale}). The key idea is to compute the density function of the truncation $x_{N,M}(t,\omega)$ given in (\ref{xNM}) taking advantage of Lemma \ref{lema_abscont}, and then proving that it converges to a density of $x(t,\omega)$. The following subsections are divided taking into account the way Lemma \ref{lema_abscont} is applied. As it shall be seen later, our approach is based upon which variable is essentially isolated when computing the inverse of the specific transformation mapping that will be chosen to apply Lemma \ref{lema_abscont}. For instance, in Subsection \ref{su1} we will isolate the random variable $x_0$ and in Subsection \ref{su3} we will isolate the random variable $\eta_1$. This permits having different hypotheses under which a density function of $x(t,\omega)$ can be approximated. This approach allows us to achieve a lot of generality in our findings (see Theorem \ref{teor1}, Theorem \ref{teor3}, Theorem \ref{teornou1} and Theorem \ref{teornou3}). We will study the homogeneous and non-homogeneous cases, corresponding to $b=0$ and $b\neq0$, respectively. The particular case $b=0$ permits having interesting and particular results for the random non-autonomous homogeneous linear differential equation (see Subsection \ref{su2}, Subsection \ref{su4} and Subsection \ref{su5}, Theorem \ref{teor2}, Theorem \ref{teorb0L}, Theorem \ref{teornou2}, Theorem \ref{avo} and Theorem \ref{teornoub0L}).

\subsection{Obtaining the density function when \texorpdfstring{$f_0$}{f0} is Lipschitz on $\mathbb{R}$} \label{su1} \ \\
 
Using Lemma \ref{lema_abscont}, we are able to compute the density of $x_{N,M}(t,\omega)$. Indeed, with the notation of Lemma \ref{lema_abscont},
\[ g(x_0,\xi_1,\ldots,\xi_N,\eta_1,\ldots,\eta_M)=\left(\e^{K_a(t,\pmb{\xi}_N)}\left\{x_0+\int_{t_0}^t S_b(s,\pmb{\eta}_M)\e^{-K_a(s,\pmb{\xi}_N)}\,\dif s\right\},\pmb{\xi}_N,\pmb{\eta}_M\right), \]
$D=\mathbb{R}^{N+M+1}$, $g(D)=\mathbb{R}^{N+M+1}$,
\[ h(x_0,\xi_1,\ldots,\xi_N,\eta_1,\ldots,\eta_M)=\left(x_0\e^{-K_a(t,\pmb{\xi}_N)}-\int_{t_0}^t S_b(s,\pmb{\eta}_M)\e^{-K_a(s,\pmb{\xi}_N)}\,\dif s,\pmb{\xi}_N,\pmb{\eta}_M\right) \]
and 
\[ Jh(x_0,\xi_1,\ldots,\xi_N,\eta_1,\ldots,\eta_M)=\e^{-K_a(t,\pmb{\xi}_N)}>0. \]
Then, taking the marginal distributions with respect to $(\pmb{\xi}_N,\pmb{\eta}_M)$ and denoting by $f_1^{N,M}(x,t)$ the density of $x_{N,M}(t,\omega)$, we have
\footnotesize
\[ f_1^{N,M}(x,t)=\int_{\mathbb{R}^{N+M}} f_0\left(x\,\e^{-K_a(t,\pmb{\xi}_N)}-\int_{t_0}^t S_b(s,\pmb{\eta}_M)\e^{-K_a(s,\pmb{\xi}_N)}\,\dif s\right)f_{\pmb{\xi}_N,\pmb{\eta}_M}(\pmb{\xi}_N,\pmb{\eta}_M)\e^{-K_a(t,\pmb{\xi}_N)}\,\dif \pmb{\xi}_N\,\dif \pmb{\eta}_M. \]
\normalsize

In order to be able to compute the limit of $f_1^{N,M}(x,t)$ when $N,M\rightarrow\infty$ easily, without loss of generality we will take $N=M$, so that we work with the density $f_1^N(x,t)$ of $x_{N,N}(t,\omega)$:
\footnotesize
\begin{equation}
 f_1^N(x,t)=\int_{\mathbb{R}^{2N}} f_0\left(x\,\e^{-K_a(t,\pmb{\xi}_N)}-\int_{t_0}^t S_b(s,\pmb{\eta}_N)\e^{-K_a(s,\pmb{\xi}_N)}\,\dif s\right)f_{\pmb{\xi}_N,\pmb{\eta}_N}(\pmb{\xi}_N,\pmb{\eta}_N)\e^{-K_a(t,\pmb{\xi}_N)}\,\dif \pmb{\xi}_N\,\dif \pmb{\eta}_N. 
\label{f1n}
\end{equation}
\normalsize

In the following theorem we establish the hypotheses under which $\{f_1^N(x,t)\}_{N=1}^\infty$ converges to a density of the solution $x(t,\omega)$ given by (\ref{sol_ale}).

\begin{theorem} \label{teor1}
Assume the following four hypotheses:
\begin{align*}
 \text{H1}: {} & \;a,b\in \leb^2([t_0,T]\times\Omega); \\
 \text{H2}: {} & \;x_0 \text{ and }(\xi_1,\ldots,\xi_N,\eta_1,\ldots,\eta_N) \text{ are absolutely continuous and independent, }N\geq1; \\
 \text{H3}: {} & \;\text{the density function of }x_0\text{, }f_0 \text{, is Lipschitz on }\mathbb{R}; \\
 \text{H4}: {} & \;\text{there exist }2\leq p\leq\infty\text{ and }4\leq q\leq\infty\text{ such that }1/p+2/q=1/2, \\
{} & \;\|\mu_b\|_{\leb^p(t_0,T)}+\sum_{j=1}^\infty\sqrt{\gamma_j}\,\|\psi_j\|_{\leb^p(t_0,T)}\|\eta_j\|_{\leb^p(\Omega)}<\infty \text{ and } \\
{} & \;\|\e^{-K_a(t,\pmb{\xi}_N)}\|_{\leb^q(\Omega)}\leq C,\text{ for all }N\geq1\text{ and }t\in [t_0,T]. 
\end{align*}
Then the sequence $\{f_1^N(x,t)\}_{N=1}^\infty$ given in (\ref{f1n}) converges in $\leb^\infty(J\times [t_0,T])$ for every bounded set $J\subseteq\mathbb{R}$, to a density $f_1(x,t)$ of the solution $x(t,\omega)$ given in (\ref{sol_ale}).
\end{theorem}

\begin{proof}
We prove that $\{f_1^N(x,t)\}_{N=1}^\infty$ is Cauchy in $\leb^\infty(J\times [t_0,T])$ for every bounded set $J\subseteq\mathbb{R}$. Fix two indexes $N>M$. 

First of all, notice that, taking the marginal distribution,
\[ f_{\pmb{\xi}_M,\pmb{\eta}_M}(\pmb{\xi}_M,\pmb{\eta}_M)=\int_{\mathbb{R}^{2(N-M)}} f_{\pmb{\xi}_N,\pmb{\eta}_N}(\pmb{\xi}_N,\pmb{\eta}_N)\,\dif \xi_{M+1}\cdots \dif\xi_N\,\dif \eta_{M+1}\cdots \dif\eta_N, \]
so, according to (\ref{f1n}) with index $M$,
\small
\begin{align}
{} & f_1^M(x,t) \nonumber \\
= {} & \int_{\mathbb{R}^{2M}} f_0\left(x\,\e^{-K_a(t,\pmb{\xi}_M)}-\int_{t_0}^t S_b(s,\pmb{\eta}_M)\e^{-K_a(s,\pmb{\xi}_M)}\,\dif s\right)f_{\pmb{\xi}_M,\pmb{\eta}_M}(\pmb{\xi}_M,\pmb{\eta}_M)\e^{-K_a(t,\pmb{\xi}_M)}\,\dif \pmb{\xi}_M\,\dif \pmb{\eta}_M \nonumber \\
= {} & \int_{\mathbb{R}^{2N}} f_0\left(x\,\e^{-K_a(t,\pmb{\xi}_M)}-\int_{t_0}^t S_b(s,\pmb{\eta}_M)\e^{-K_a(s,\pmb{\xi}_M)}\,\dif s\right)f_{\pmb{\xi}_N,\pmb{\eta}_N}(\pmb{\xi}_N,\pmb{\eta}_N)\e^{-K_a(t,\pmb{\xi}_M)}\,\dif \pmb{\xi}_N\,\dif \pmb{\eta}_N. \label{f1m_marginal}
\end{align}
\normalsize

Using (\ref{f1n}) and (\ref{f1m_marginal}), we start to estimate
\small
\begin{align*}
& |f_1^N(x,t)-f_1^M(x,t)|\leq \int_{\mathbb{R}^{2N}}\bigg\{\,\bigg|f_0\left(x\,\e^{-K_a(t,\pmb{\xi}_N)}-\int_{t_0}^t S_b(s,\pmb{\eta}_N)\e^{-K_a(s,\pmb{\xi}_N)}\,\dif s\right)\e^{-K_a(t,\pmb{\xi}_N)} \\
- {} & f_0\left(x\,\e^{-K_a(t,\pmb{\xi}_M)}-\int_{t_0}^t S_b(s,\pmb{\eta}_M)\e^{-K_a(s,\pmb{\xi}_M)}\,\dif s\right)\e^{-K_a(t,\pmb{\xi}_M)}\bigg| f_{\pmb{\xi}_N,\pmb{\eta}_N}(\pmb{\xi}_N,\pmb{\eta}_N)\bigg\}\,\dif \pmb{\xi}_N\,\dif \pmb{\eta}_N  \\
\leq {} & \int_{\mathbb{R}^{2N}} \bigg\{f_0\left(x\,\e^{-K_a(t,\pmb{\xi}_N)}-\int_{t_0}^t S_b(s,\pmb{\eta}_N)\e^{-K_a(s,\pmb{\xi}_N)}\,\dif s\right) \\
\cdot {} & |\e^{-K_a(t,\pmb{\xi}_N)}-\e^{-K_a(t,\pmb{\xi}_M)}|f_{\pmb{\xi}_N,\pmb{\eta}_N}(\pmb{\xi}_N,\pmb{\eta}_N)\bigg\}\,\dif \pmb{\xi}_N\,\dif \pmb{\eta}_N \\
+ {} & \int_{\mathbb{R}^{2N}}\bigg\{\,\bigg|f_0\left(x\,\e^{-K_a(t,\pmb{\xi}_N)}-\int_{t_0}^t S_b(s,\pmb{\eta}_N)\e^{-K_a(s,\pmb{\xi}_N)}\,\dif s\right)\\
- {} & f_0\left(x\,\e^{-K_a(t,\pmb{\xi}_M)}-\int_{t_0}^t S_b(s,\pmb{\eta}_M)\e^{-K_a(s,\pmb{\xi}_M)}\,\dif s\right)\bigg| \e^{-K_a(t,\pmb{\xi}_M)}f_{\pmb{\xi}_N,\pmb{\eta}_N}(\pmb{\xi}_N,\pmb{\eta}_N)\bigg\}\,\dif \pmb{\xi}_N\,\dif \pmb{\eta}_N \\
\stackrel{\Delta}{=} {} & \mathrm{(I1)}+\mathrm{(I2)}.
\end{align*}
\normalsize

Henceforth, concerning notation, we will denote by $C$ any constant independent of $N$, $t$ and $x$, so that the notation will not become cumbersome. Call $L$ the Lipschitz constant of $f_0$.

Now we introduce two inequalities that are direct consequence of Cauchy-Schwarz inequality and that will play a crucial role later on:
\begin{align}
\| K_a(t,\pmb{\xi}_N)-K_a(t,\pmb{\xi}_M)\|_{\leb^2(\Omega)}= {} & \mathbb{E}\left[\left(\int_{t_0}^t (a_N(s)-a_M(s))\,\dif s\right)^2\right]^{\frac12}  \nonumber \\
\leq {} & \sqrt{t-t_0}\,\mathbb{E}\left[\int_{t_0}^t (a_N(s)-a_M(s))^2\,\dif s\right]^{\frac12} \nonumber \\
\leq {} & C\,\|a_N-a_M\|_{\leb^2([t_0,T]\times\Omega)}, \label{ka}
\end{align}
\begin{equation}
 \int_{t_0}^T \mathbb{E}[ |S_b(s,\pmb{\eta}_N)-S_b(s,\pmb{\eta}_M)|^2]^{\frac12}\,\dif s\leq C\,\|b_N-b_M\|_{\leb^2([t_0,T]\times\Omega)}. 
 \label{sb}
\end{equation}

We bound $\mathrm{(I1)}$. Since $f_0$ is Lipschitz continuous (therefore uniformly continuous) and integrable on $\mathbb{R}$, it is bounded on the real line\footnote{If a function $f$ belongs to $\leb^1(\mathbb{R})$ and is uniformly continuous on $\mathbb{R}$, then $\lim_{x\rightarrow\infty} f(x)=0$, and as a consequence $f$ is bounded on $\mathbb{R}$. Indeed, suppose that $\lim_{x\rightarrow\infty} f(x)\neq0$. By definition, there is an $\epsilon_0>0$ and a sequence $\{x_n\}_{n=1}^\infty$ increasing to $\infty$ such that $|f(x_n)|>\epsilon_0$. We may assume $x_{n+1}-x_n>1$, for $n\geq1$. By uniform continuity, there exists a $\delta=\delta(\epsilon_0)>0$ such that $|f(x)-f(y)|<\epsilon_0/2$, if $|x-y|<\delta$. We may assume $0<\delta<1/2$, so that $\{(x_n-\delta,x_n+\delta)\}_{n=1}^\infty$ are pairwise disjoint intervals. We have $|f(x)|>\epsilon_0/2$ for all $x\in (x_n-\delta,x_n+\delta)$ and every $n\in\mathbb{N}$. Thereby, $\int_{\mathbb{R}}|f(x)|\,\dif x\geq \sum_{n=1}^\infty \int_{x_n-\delta}^{x_n+\delta}|f(x)|\,\dif x\geq\sum_{n=1}^\infty \delta\epsilon_0=\infty$, which is a contradiction.}, therefore 
\[ f_0\left(x\,\e^{-K_a(t,\pmb{\xi}_N)}-\int_{t_0}^t S_b(s,\pmb{\eta}_N)\e^{-K_a(s,\pmb{\xi}_N)}\,\dif s\right)\leq C. \]
To bound $|\e^{-K_a(t,\pmb{\xi}_N)}-\e^{-K_a(t,\pmb{\xi}_M)}|$, we use the Mean Value Theorem to the real function $\e^{-x}$, for each fixed $t$ and $\pmb{\xi}_N$. We have
\[ \e^{-K_a(t,\pmb{\xi}_N)}-\e^{-K_a(t,\pmb{\xi}_M)}=-\e^{-\delta_{t,\pmb{\xi}_N}}\{K_a(t,\pmb{\xi}_N)-K_a(t,\pmb{\xi}_M)\}, \]
where $K_a(t,\pmb{\xi}_N)\leq \delta_{t,\pmb{\xi}_N}\leq K_a(t,\pmb{\xi}_M)$ or $K_a(t,\pmb{\xi}_M)\leq \delta_{t,\pmb{\xi}_N}\leq K_a(t,\pmb{\xi}_N)$, which implies 
\[ \e^{-\delta_{t,\pmb{\xi}_N}}\leq \e^{-K_a(t,\pmb{\xi}_N)}+\e^{-K_a(t,\pmb{\xi}_M)}. \]
Thus,
\begin{equation}
 |\e^{-K_a(t,\pmb{\xi}_N)}-\e^{-K_a(t,\pmb{\xi}_M)}|\leq (\e^{-K_a(t,\pmb{\xi}_N)}+\e^{-K_a(t,\pmb{\xi}_M)})|K_a(t,\pmb{\xi}_N)-K_a(t,\pmb{\xi}_M)|. 
\label{tvm}
\end{equation}
We bound
\begin{align*}
{} & f_0\left(x\,\e^{-K_a(t,\pmb{\xi}_N)}-\int_{t_0}^t S_b(s,\pmb{\eta}_N)\e^{-K_a(s,\pmb{\xi}_N)}\,\dif s\right)|\e^{-K_a(t,\pmb{\xi}_N)}-\e^{-K_a(t,\pmb{\xi}_M)}| \\
\leq {} & C(\e^{-K_a(t,\pmb{\xi}_N)}+\e^{-K_a(t,\pmb{\xi}_M)})|K_a(t,\pmb{\xi}_N)-K_a(t,\pmb{\xi}_M)|. 
\end{align*}
Then $\mathrm{(I1)}$ is bounded by the expectation of this last expression:
\begin{align*}
\mathrm{(I1)}= {} & \int_{\mathbb{R}^{2N}} \bigg\{f_0\left(x\,\e^{-K_a(t,\pmb{\xi}_N)}-\int_{t_0}^t S_b(s,\pmb{\eta}_N)\e^{-K_a(s,\pmb{\xi}_N)}\,\dif s\right) \\
\cdot {} & |\e^{-K_a(t,\pmb{\xi}_N)}-\e^{-K_a(t,\pmb{\xi}_M)}|f_{\pmb{\xi}_N,\pmb{\eta}_N}(\pmb{\xi}_N,\pmb{\eta}_N)\bigg\}\,\dif \pmb{\xi}_N\,\dif \pmb{\eta}_N \\
\leq {} & C\,\mathbb{E}\left[(\e^{-K_a(t,\pmb{\xi}_N)}+\e^{-K_a(t,\pmb{\xi}_M)})|K_a(t,\pmb{\xi}_N)-K_a(t,\pmb{\xi}_M)|\right]. 
\end{align*}
By Hölder's Inequality, hypothesis H4 and bound (\ref{ka}),
\begin{align*}
 \mathrm{(I1)}\leq {} & C(\|\e^{-K_a(t,\pmb{\xi}_N)}\|_{\leb^2(\Omega)}+\|\e^{-K_a(t,\pmb{\xi}_M)}\|_{\leb^2(\Omega)})\|K_a(t,\pmb{\xi}_N)-K_a(t,\pmb{\xi}_M)\|_{\leb^2(\Omega)} \\
\leq {} & C\,\|K_a(t,\pmb{\xi}_N)-K_a(t,\pmb{\xi}_M)\|_{\leb^2(\Omega)}\leq C\,\|a_N-a_M\|_{\leb^2([t_0,T]\times \Omega)}.
\end{align*}

We bound $\mathrm{(I2)}$. Using the Lipschitz condition and the triangular inequality, 
\begin{align*}
{} & \bigg|f_0\left(x\,\e^{-K_a(t,\pmb{\xi}_N)}-\int_{t_0}^t S_b(s,\pmb{\eta}_N)\e^{-K_a(s,\pmb{\xi}_N)}\,\dif s\right)\\
- {} & f_0\left(x\,\e^{-K_a(t,\pmb{\xi}_M)}-\int_{t_0}^t S_b(s,\pmb{\eta}_M)\e^{-K_a(s,\pmb{\xi}_M)}\,\dif s\right)\bigg| \\
\leq {} & L|x|\,|\e^{-K_a(t,\pmb{\xi}_N)}-\e^{-K_a(t,\pmb{\xi}_M)}|+L\int_{t_0}^t |S_b(s,\pmb{\eta}_N)\e^{-K_a(s,\pmb{\xi}_N)}-S_b(s,\pmb{\eta}_M)\e^{-K_a(s,\pmb{\xi}_M)}|\,\dif s\\
\leq {} & L|x|\,|\e^{-K_a(t,\pmb{\xi}_N)}-\e^{-K_a(t,\pmb{\xi}_M)}|+L\int_{t_0}^t |S_b(s,\pmb{\eta}_N)|\,|\e^{-K_a(s,\pmb{\xi}_N)}-\e^{-K_a(s,\pmb{\xi}_M)}|\,\dif s \\
+ {} & L\int_{t_0}^t \e^{-K_a(s,\pmb{\xi}_M)} |S_b(s,\pmb{\eta}_N)-S_b(s,\pmb{\eta}_M)|\,\dif s\\
\stackrel{\Delta}{=} {} & \mathrm{(B1)}+\mathrm{(B2)}+\mathrm{(B3)}.
\end{align*}
Using the bound (\ref{tvm}) from the Mean Value Theorem, 
\[\mathrm{(B1)}\leq L|x|(\e^{-K_a(t,\pmb{\xi}_N)}+\e^{-K_a(t,\pmb{\xi}_M)})|K_a(t,\pmb{\xi}_N)-K_a(t,\pmb{\xi}_M)|\]
and
\[\mathrm{(B2)}\leq L\int_{t_0}^T |S_b(s,\pmb{\eta}_N)|(\e^{-K_a(s,\pmb{\xi}_N)}+\e^{-K_a(s,\pmb{\xi}_M)})|K_a(s,\pmb{\xi}_N)-K_a(s,\pmb{\xi}_M)|\,\dif s. \]
Since 
\small
\begin{align*}
\mathrm{(I2)}= {} & \int_{\mathbb{R}^{2N}}\bigg\{\,\bigg|f_0\left(x\,\e^{-K_a(t,\pmb{\xi}_N)}-\int_{t_0}^t S_b(s,\pmb{\eta}_N)\e^{-K_a(s,\pmb{\xi}_N)}\,\dif s\right)\\
- {} & f_0\left(x\,\e^{-K_a(t,\pmb{\xi}_M)}-\int_{t_0}^t S_b(s,\pmb{\eta}_M)\e^{-K_a(s,\pmb{\xi}_M)}\,\dif s\right)\bigg| \e^{-K_a(t,\pmb{\xi}_M)}f_{\pmb{\xi}_N,\pmb{\eta}_N}(\pmb{\xi}_N,\pmb{\eta}_N)\bigg\}\,\dif \pmb{\xi}_N\,\dif \pmb{\eta}_N \\
\leq {} & \mathbb{E}[\mathrm{(B1)}\cdot \e^{-K_a(t,\pmb{\xi}_M)}]+\mathbb{E}[\mathrm{(B2)}\cdot \e^{-K_a(t,\pmb{\xi}_M)}]+\mathbb{E}[\mathrm{(B3)}\cdot \e^{-K_a(t,\pmb{\xi}_M)}] \\
\stackrel{\Delta}{=} {} & \mathrm{(E1)}+\mathrm{(E2)}+\mathrm{(E3)}, 
\end{align*}
\normalsize
we need to bound these three expectations $\mathrm{(E1)}$, $\mathrm{(E2)}$ and $\mathrm{(E3)}$. First, for $\mathrm{(E1)}$, using Hölder's Inequality, hypothesis H4 and (\ref{ka}), one gets 
\small
\begin{align*}
{} & \mathrm{(E1)} \\
\leq {} & L|x|\,\| \e^{-K_a(t,\pmb{\xi}_M)}\|_{\leb^4(\Omega)}(\|\e^{-K_a(t,\pmb{\xi}_N)}\|_{\leb^4(\Omega)}+\|\e^{-K_a(t,\pmb{\xi}_M)}\|_{\leb^4(\Omega)})\|K_a(t,\pmb{\xi}_N)-K_a(t,\pmb{\xi}_M)\|_{\leb^2(\Omega)} \\
\leq {} & C|x|\,\|K_a(t,\pmb{\xi}_N)-K_a(t,\pmb{\xi}_M)\|_{\leb^2(\Omega)} \leq C|x|\,\|a_N-a_M\|_{\leb^2([t_0,T]\times \Omega)}.
\end{align*}
\normalsize
By an analogous reasoning, but also using (\ref{sb}), one deduces the following bounds:
\begin{align*}
{} & \mathrm{(E2)}  \\
\leq {} & L\int_{t_0}^T \mathbb{E}[|S_b(s,\pmb{\eta}_N)|^p]^{\frac{1}{p}}\,\mathbb{E}[\e^{-q\,K_a(t,\pmb{\xi}_M)}]^{\frac{1}{q}}(\mathbb{E}[\e^{-q\,K_a(s,\pmb{\xi}_N)}]^{\frac{1}{q}}+\mathbb{E}[\e^{-q\,K_a(s,\pmb{\xi}_M)}]^{\frac{1}{q}}) \\
\cdot {} & \mathbb{E}[|K_a(s,\pmb{\xi}_N)-K_a(s,\pmb{\xi}_M)|^2]^{\frac12}\,\dif s \\
\leq {} & C\,\int_{t_0}^T \mathbb{E}[|S_b(s,\pmb{\eta}_N)|^p]^{\frac{1}{p}}\|K_a(s,\pmb{\xi}_N)-K_a(s,\pmb{\xi}_M)\|_{\leb^2(\Omega)}\,\dif s \\
\leq {} & C\,\|a_N-a_M\|_{\leb^2([t_0,T]\times \Omega)} \int_{t_0}^T \mathbb{E}[|S_b(s,\pmb{\eta}_N)|^p]^{\frac{1}{p}}\,\dif s \\
\leq {} & C\, \|a_N-a_M\|_{\leb^2([t_0,T]\times \Omega)} \| S_b(t,\pmb{\eta}_N(\omega))\|_{\leb^p([t_0,T]\times\Omega)} \\
\leq {} & C\,\|a_N-a_M\|_{\leb^2([t_0,T]\times \Omega)}
\end{align*}
and
\begin{align*}
\mathrm{(E3)}= {} & L\,\int_{t_0}^T\mathbb{E}[\e^{-K_a(t,\pmb{\xi}_M)}\e^{-K_a(s,\pmb{\xi}_M)}|S_b(s,\pmb{\eta}_N)-S_b(s,\pmb{\eta}_M)|]\,\dif s\\
\leq {} & L\,\int_{t_0}^T \|\e^{-K_a(t,\pmb{\xi}_M)}\|_{\leb^4(\Omega)}\|\e^{-K_a(s,\pmb{\xi}_M)}\|_{\leb^4(\Omega)} \|S_b(s,\pmb{\eta}_N)-S_b(s,\pmb{\eta}_M)\|_{\leb^2(\Omega)}\,\dif s\\
\leq {} & C\,\int_{t_0}^T\|S_b(s,\pmb{\eta}_N)-S_b(s,\pmb{\eta}_M)\|_{\leb^2(\Omega)}\,\dif s\\
\leq {} & C\,\|b_N-b_M\|_{\leb^2([t_0,T]\times \Omega)}.
\end{align*}
Thus,
\begin{align*}
 \mathrm{(I2)}\leq {} & \mathrm{(E1)}+\mathrm{(E2)}+\mathrm{(E3)} \\
\leq {} & C(|x|+1)\|a_N-a_M\|_{\leb^2([t_0,T]\times \Omega)}+C\,\|b_N-b_M\|_{\leb^2([t_0,T]\times \Omega)}. 
\end{align*}

Since $\|a_N-a_M\|_{\leb^2([t_0,T]\times \Omega)}\rightarrow0$ and $\|b_N-b_M\|_{\leb^2([t_0,T]\times \Omega)}\rightarrow0$ when $N,M\rightarrow\infty$, the sequence $\{f_1^N(x,t)\}_{N=1}^\infty$ is Cauchy in $\leb^\infty(J\times [t_0,T])$ for every bounded set $J\subseteq\mathbb{R}$.

Let 
\[ g(x,t)=\lim_{N\rightarrow\infty} f_1^N(x,t), \]
$x\in\mathbb{R}$ and $t\in [t_0,T]$. Let us see that $x(t,\cdot)$ is absolutely continuous and $g(\cdot,t)$ is a density of $x(t,\cdot)$.

First, notice that $g(\cdot,t)\in \leb^1(\mathbb{R})$, since by Fatou's Lemma \cite[Lemma 1.7, p.61]{stein},
\[ \int_{\mathbb{R}} g(x,t)\,\dif x=\int_{\mathbb{R}} \lim_{N\rightarrow\infty}f_1^N(x,t)\,\dif x\leq \liminf_{N\rightarrow\infty} \underbrace{\int_{\mathbb{R}} f_1^N(x,t)\,\dif x}_{=1}=1<\infty. \]

Recall that 
\[ x_{N,N}(t,\omega)=x_0(\omega)\,\e^{\int_{t_0}^t a_N(s,\omega)\,\dif s}+\int_{t_0}^t b_N(s,\omega)\,\e^{\int_s^t a_N(r,\omega)\,\dif r}\,\dif s. \]
We check that $x_{N,N}(t,\omega)\stackrel{N\rightarrow\infty}{\longrightarrow} x(t,\omega)$ for every $t\in [t_0,T]$ and a.e. $\omega \in \Omega$. 

We know that $a_N(\cdot,\omega)\rightarrow a(\cdot,\omega)$ in $\leb^2([t_0,T])$ and $b_N(\cdot,\omega)\rightarrow b(\cdot,\omega)$ in $\leb^2([t_0,T])$ as $N\rightarrow\infty$, for a.e. $\omega \in \Omega$, because the Fourier series converges in $\leb^2$.

On the one hand, $\int_{t_0}^t a_N(s,\omega)\,\dif s \stackrel{N\rightarrow\infty}{\longrightarrow} \int_{t_0}^t a(s,\omega)\,\dif s$ for all $t\in [t_0,T]$ and for a.e. $\omega \in \Omega$, whence 
\begin{equation}
x_0(\omega)\e^{\int_{t_0}^t a_N(s,\omega)\,\dif s}\stackrel{N\rightarrow\infty}{\longrightarrow} x_0(\omega)\e^{\int_{t_0}^t a(s,\omega)\,\dif s},
 \label{lima}
\end{equation}
for all $t\in [t_0,T]$ and for a.e. $\omega \in \Omega$.

On the other hand,
\begin{align*}
 \left|b_N(s,\omega)\,\e^{\int_s^t a_N(r,\omega)\,\dif r}-b(s,\omega)\,\e^{\int_s^t a(r,\omega)\,\dif r}\right|\leq {} & |b_N(s,\omega)-b(s,\omega)|\e^{\int_s^t a_N(r,\omega)\,\dif r}\\
+ {} & |b(s,\omega)|\left|\e^{\int_s^t a_N(r,\omega)\,\dif r}-\e^{\int_s^t a(r,\omega)\,\dif r}\right|.
\end{align*}
We bound the expressions involving exponentials. First, using the deterministic Cauchy-Schwarz inequality for integrals, one gets
\[ \e^{\int_s^t a_N(r,\omega)\,\dif r}\leq \e^{\sqrt{T-t_0}\,\|a_N(\cdot,\omega)\|_{\leb^2([t_0,T])}}\leq \e^{C_\omega}=C_\omega,  \]
where $C_\omega$ represents a constant depending on $\omega$, and independent of $N$, $t$ and $x$. By the Mean Value Theorem applied to the real function $\e^x$,
\[ \e^{\int_s^t a_N(r,\omega)\,\dif r}-\e^{\int_s^t a(r,\omega)\,\dif r}=\e^{\delta_{N,s,t,\omega}}\left(\int_s^t a_N(r,\omega)\,\dif r-\int_s^t a(r,\omega)\,\dif r\right), \]
where
\begin{align*}
|\delta_{N,s,t,\omega}|\leq {} & \max\left\{\left|\int_s^t a_N(r,\omega)\,\dif r\right|,\left|\int_s^t a(r,\omega)\,\dif r\right|\right\}\\
\leq {} & \sqrt{T-t_0}\, \max\left\{ \|a(\cdot,\omega)\|_{\leb^2([t_0,T])},\|a_N(\cdot,\omega)\|_{\leb^2([t_0,T])}\right\}\leq C_\omega.
\end{align*}
Thus,
\begin{align*}
 \left|\e^{\int_s^t a_N(r,\omega)\,\dif r}-\e^{\int_s^t a(r,\omega)\,\dif r}\right| \leq {} & C_\omega \left|\int_s^t a_N(r,\omega)\,\dif r-\int_s^t a(r,\omega)\,\dif r\right|\\
\leq {} & C_\omega \|a_N(\cdot,\omega)-a(\cdot,\omega)\|_{\leb^2([t_0,T])}.
\end{align*}

Therefore,
\begin{align}
{} & \int_{t_0}^t \left|b_N(s,\omega)\,\e^{\int_s^t a_N(r,\omega)\,\dif r}-b(s,\omega)\,\e^{\int_s^t a(r,\omega)\,\dif r}\right|\,\dif s  \nonumber \\
\leq {} & C_\omega\left\{ \|b_N(\cdot,\omega)-b(\cdot,\omega)\|_{\leb^1([t_0,T])}+\|b(\cdot,\omega)\|_{\leb^1([t_0,T])}\|a_N(\cdot,\omega)-a(\cdot,\omega)\|_{\leb^2([t_0,T])}\right\} \nonumber \\
\stackrel{N\rightarrow\infty}{\longrightarrow} {} & 0. \label{limb}
\end{align}

This shows that $x_{N,N}(t,\omega)\rightarrow x(t,\omega)$ as $N\rightarrow\infty$ for every $t\in [t_0,T]$ and a.e. $\omega \in \Omega$. This says that $x_{N,N}(t,\cdot)\rightarrow x(t,\cdot)$ converges a.s. as $N\rightarrow\infty$, therefore there is convergence in law:
\[\lim_{N\rightarrow\infty} F_N(x,t)=F(x,t),\]
for every $x \in \mathbb{R}$ which is a point of continuity of $F(\cdot,t)$, where $F_N(\cdot,t)$ and $F(\cdot,t)$ are the distribution functions of $x_{N,N}(t,\cdot)$ and $x(t,\cdot)$, respectively. Since $f_1^N(x,t)$ is the density of $x_{N,N}(t,\omega)$,
\begin{equation}
 F_N(x,t)=F_N(x_0,t)+\int_{x_0}^x f_1^N(y,t)\,\dif y. 
 \label{FTCD}
\end{equation}
If $x$ and $x_0$ are points of continuity of $F(\cdot,t)$, taking limits when $N\rightarrow\infty$ we get
\begin{equation}
 F(x,t)=F(x_0,t)+\int_{x_0}^x g(y,t)\,\dif y 
 \label{FTCD2}
\end{equation}
(recall that $\{f_1^N(x,t)\}_{N=1}^\infty$ converges to $g(x,t)$ in $\leb^{\infty}(J\times \mathbb{R})$ for every bounded set $J\subseteq \mathbb{R}$, so we can interchange the limit and the integral). As the points of discontinuity of $F(\cdot,t)$ are countable and $F(\cdot,t)$ is right continuous, we obtain 
\[ F(x,t)=F(x_0,t)+\int_{x_0}^x g(y,t)\,\dif y \]
for all $x_0$ and $x$ in $\mathbb{R}$.

Thus, $g(x,t)=f_1(x,t)$ is a density for $x(t,\omega)$, as wanted.
\end{proof}

\subsection{Obtaining the density function when $b=0$ and \texorpdfstring{$f_0$}{f0} is Lipchitz on $\mathbb{R}$} \label{su2} \ \\

If $b=0$, all our previous exposition can be adapted to approximate the density function of the solution of the randomized non-autonomous homogeneous linear differential equation associated to the initial value problem (\ref{edo_det}). In this case, the solution stochastic process is 
\begin{equation}
 x(t,\omega)=x_0(\omega)\,\e^{\int_{t_0}^t a(s,\omega)\,\dif s}. 
\label{xhomo}
\end{equation}
We only need the Karhunen-Loève expansion of the stochastic process $a$,
\[ a(t,\omega)=\mu_a(t)+\sum_{j=1}^\infty \sqrt{\nu_j}\,\phi_j(t)\xi_j(\omega), \]
where $\{(\nu_j,\phi_j)\}_{j=1}^\infty$ are the corresponding pairs of eigenvalues and eigenfunctions and $\{\xi_j\}_{j=1}^\infty$ are random variables with zero expectation, unit variance and pairwise uncorrelated. Truncating $a$ via $a_N(t,\omega)=\mu_a(t)+\sum_{j=1}^N \sqrt{\nu_j}\,\phi_j(t)\xi_j(\omega)$, we obtain a truncation for the solution,
\begin{equation}
 x_{N}(t,\omega)= x_0(\omega)\,\e^{\int_{t_0}^t a_N(s,\omega)\,\dif s}. 
\label{truncamm}
\end{equation}
The density function of $x_N(t,\omega)$ is given by
\begin{equation}
 f_1^N(x,t)=\int_{\mathbb{R}^{N}} f_0\left(x\,\e^{-K_a(t,\pmb{\xi}_N)}\right)f_{\pmb{\xi}_N}(\pmb{\xi}_N)\e^{-K_a(t,\pmb{\xi}_N)}\,\dif \pmb{\xi}_N 
\label{f1Nhomo}
\end{equation}
(see (\ref{f1n}) with $b=0$). 

Notice that the evaluation inside $f_0$ in (\ref{f1Nhomo}), $x\,\e^{-K_a(t,\pmb{\xi}_N)}$, has the same sign as $x$. This is different to (\ref{f1n}). We define
\begin{equation} 
D(x_0)=\begin{cases} (0,\infty),\,& \text{if }x_0(\omega)>0 \text{ a.s.} \\ (-\infty,0),\,& \text{if }x_0(\omega)<0 \text{ a.s.} \\ \mathbb{R},\,&\text{otherwise}. \end{cases}
 \label{Dx0}
\end{equation}
Theorem \ref{teor1} becomes: 
\begin{theorem} \label{teor2}
Assume that
\begin{align*}
 \text{H1}: {} & \;a\in \leb^2([t_0,T]\times\Omega); \\
 \text{H2}: {} & \;x_0 \text{ and }(\xi_1,\ldots,\xi_N) \text{ are absolutely continuous and independent, }N\geq1; \\
 \text{H3}: {} & \;\text{the density function of }x_0\text{, }f_0 \text{, is Lipschitz on }D(x_0); \\
 \text{H4}: {} & \;\|\e^{-K_a(t,\pmb{\xi}_N)}\|_{\leb^4(\Omega)}\leq C,\text{ for all }N\geq1\text{ and }t\in [t_0,T] \\ {} & \;(p=\infty \text{ and }q=4\text{ in Theorem \ref{teor1}}).
\end{align*}
Then the sequence $\{f_1^N(x,t)\}_{N=1}^\infty$ given in (\ref{f1Nhomo}) converges in $\leb^\infty(J\times [t_0,T])$ for every bounded set $J\subseteq\mathbb{R}$, to a density $f_1(x,t)$ of the solution (\ref{xhomo}).
\end{theorem}

\subsection{Obtaining the density function when \texorpdfstring{$b(t,\cdot)$}{b(t,·)} is not a constant random variable and \texorpdfstring{$f_{\eta_1}$}{f_eta1} is Lipchitz on $\mathbb{R}$} \label{su3} \ \\

Take truncation (\ref{xNM}) of the solution $x(t,\omega)$ given in (\ref{sol_ale}), with $N=M$:
\[ x_{N,N}(t,\omega)=x_0(\omega)\e^{K_a(t,\pmb{\xi}_N(\omega))}+\int_{t_0}^t S_b(s,\pmb{\eta}_N(\omega))\e^{K_a(t,\pmb{\xi}_N(\omega))-K_a(s,\pmb{\xi}_N(\omega))}\,\dif s. \]
The idea is to compute (\ref{f1n}) again, but instead of isolating $x_0$ when using Lemma \ref{lema_abscont}, we will isolate $\eta_1$. This can be done whenever $b(t,\omega)\neq \mu_b(t)$ for a.e. $t\in [t_0,T]$ and a.e. $\omega\in\Omega$, since in this case the Karhunen-Loève expansion of $b$ will have more terms than the mean, in particular the first term where it appears the random variable $\eta_1$. 

To apply Lemma \ref{lema_abscont} we need to set some assumptions, as it will become clearer when writing the map of the transformation $g$ and its inverse $h$. We will assume that the random variables $\xi_1,\xi_2,\ldots$ have compact support in $[-A,A]$ ($A>0$), $\sum_{j=1}^\infty \sqrt{\nu_j}\,\left|\int_{t_0}^t\phi_j(s)\,\dif s\right|<\infty$ for all $t\in [t_0,T]$ (see Remark \ref{rmk}), and $\psi_1>0$ on $(t_0,T)$. In such a case,
\begin{align*}
\e^{-K_a(t,\pmb{\xi}_M)}= {} & \e^{-\int_{t_0}^t \mu_a(s)\,\dif s-\sum_{k=1}^M \sqrt{\nu_k}\,\left(\int_{t_0}^t \phi_k(s)\,\dif s\right)\xi_k}\\
  \geq  {} & \e^{-\int_{t_0}^t \mu_a(s)\,\dif s-\sum_{k=1}^M \sqrt{\nu_k}\,\left|\int_{t_0}^t \phi_k(s)\,\dif s\right||\xi_k|} \\
\geq {} & \e^{-\int_{t_0}^t \mu_a(s)\,\dif s-\sum_{k=1}^\infty \sqrt{\nu_k}\,\left|\int_{t_0}^t \phi_k(s)\,\dif s\right|A},
\end{align*}
whence
\begin{equation}
\int_{t_0}^t \psi_1(s)\e^{-K_a(s,\pmb{\xi}_M)}\,\dif s\geq\int_{t_0}^t \psi_1(s)\e^{-\int_{t_0}^s \mu_a(r)\,\dif r-\sum_{k=1}^\infty \sqrt{\nu_k}\,\left|\int_{t_0}^s \phi_k(r)\,\dif r\right|A}\,\dif s=:C(t)>0, 
 \label{assum}
\end{equation}
for $t\in(t_0,T]$.

Using the notation from Lemma \ref{lema_abscont}, we have
\begin{align*}
 & g(x_0,\xi_1,\ldots,\xi_N,\eta_1,\ldots,\eta_N) \\
= {} & \left(x_0,\xi_1,\ldots,\xi_N,x_0\e^{K_a(t,\pmb{\xi}_N)}+\int_{t_0}^t S_b(s,\pmb{\eta}_N)\e^{K_a(t,\pmb{\xi}_N)-K_a(s,\pmb{\xi}_N)}\,\dif s,\eta_2,\ldots,\eta_N\right), 
\end{align*}
$D=\mathbb{R}\times [-A,A]^N\times \mathbb{R}^N$ (so that $g$ is injective on $D$ by (\ref{assum})), 
\[ g(D)=\mathbb{R}\times [-A,A]^N\times \mathbb{R}^N=:\mathcal{D}_N, \]
\small
\begin{align*}
 &h(x_0,\xi_1,\ldots,\xi_N,\eta_1,\ldots,\eta_N) \\
= {} & \left(x_0,\xi_1,\ldots,\xi_N,\frac{\eta_1\e^{-K_a(t,\pmb{\xi}_N)}-x_0-\int_{t_0}^t\left(\mu_b(s)+\sum_{i=2}^N\sqrt{\gamma_i}\,\psi_i(s)\eta_i\right)\e^{-K_a(s,\pmb{\xi}_N)}\,\dif s}{\sqrt{\gamma_1}\int_{t_0}^t \psi_1(s)\e^{-K_a(s,\pmb{\xi}_N)}\,\dif s},\eta_2,\ldots,\eta_N\right) 
\end{align*}
\normalsize
and 
\[ Jh(x_0,\xi_1,\ldots,\xi_N,\eta_1,\ldots,\eta_N)=\frac{\e^{-K_a(t,\pmb{\xi}_N)}}{\sqrt{\gamma_1}\int_{t_0}^t \psi_1(s)\e^{-K_a(s,\pmb{\xi}_N)}\,\dif s}>0. \]
Note that all denominators are distinct from $0$ due to (\ref{assum}). Suppose that $x_0$, $\eta_1$ and $(\xi_1,\ldots,\xi_N,\eta_2,\ldots,\eta_N)$ are independent, for $N\geq1$. Then, taking the marginal distributions we obtain another expression for the density $f_1^N(x,t)$ of $x_N(t,\omega)$ given in (\ref{f1n}):
\small
\begin{align}
 f_1^{N}(x,t)= {} & \int_{\mathcal{D}_N} f_{\eta_1}\left(\frac{x\e^{-K_a(t,\pmb{\xi}_N)}-x_0-\int_{t_0}^t\left(\mu_b(s)+\sum_{i=2}^N\sqrt{\gamma_i}\,\psi_i(s)\eta_i\right)\e^{-K_a(s,\pmb{\xi}_N)}\,\dif s}{\sqrt{\gamma_1}\int_{t_0}^t \psi_1(s)\e^{-K_a(s,\pmb{\xi}_N)}\,\dif s}\right) \nonumber \\
\cdot {} & f_0(x_0)f_{(\xi_1,\ldots,\xi_N,\eta_2,\ldots,\eta_N)}(\xi_1,\ldots,\xi_N,\eta_2,\ldots,\eta_N) \nonumber \\
\cdot {} & \frac{\e^{-K_a(t,\pmb{\xi}_N)}}{\sqrt{\gamma_1}\int_{t_0}^t \psi_1(s)\e^{-K_a(s,\pmb{\xi}_N)}\,\dif s} \,\dif x_0\,\dif\xi_1\cdots \dif\xi_N\,\dif\eta_2\cdots \dif\eta_N. \label{f1nou}
\end{align}
\normalsize

\begin{remark} \label{rmk}
We show that the condition 
\[ \sup_{t\in[t_0,T]}\sum_{j=1}^\infty \sqrt{\nu_j}\,\left|\int_{t_0}^t\phi_j(s)\,\dif s\right|<\infty \] 
fulfills, so it is not a requirement in our development. Let $X\in \leb^2([t_0,T]\times\Omega)$ be a stochastic process. Write its Karhunen-Loève expansion as $X(t,\omega)=\mu(t)+\sum_{j=1}^\infty \sqrt{\nu_j}\,\phi_j(t)\xi_j(\omega)$.
Then 
\[ \sup_{t\in[t_0,T]}\sum_{j=1}^\infty \sqrt{\nu_j}\,\left|\int_{t_0}^t \phi_j(s)\,\dif s\right|\mathbb{E}[|\xi_j|]\leq \sup_{t\in[t_0,T]}\sum_{j=1}^\infty \sqrt{\nu_j}\,\left|\int_{t_0}^t \phi_j(s)\,\dif s\right|<\infty. \]
Indeed, as $\mathbb{E}[|\xi_j|]^2\leq \mathbb{E}[\xi_j^2]=1$, the first inequality holds. For the second inequality, first use Pythagoras's Theorem in $\leb^2([t_0,T]\times\Omega)$:
\begin{align*}
\sum_{j=M+1}^N\nu_j=\left\|\sum_{j=M+1}^N \sqrt{\nu_j}\,\phi_j\xi_j\right\|_{\leb^2([t_0,T]\times\Omega)}^2\stackrel{N,M\rightarrow\infty}{\longrightarrow}0, 
\end{align*}
therefore $\sum_{j=1}^\infty \nu_j<\infty$. By Parseval's identity for deterministic Fourier series, since $\{\phi_j\}_{j=1}^\infty$ is an orthonormal basis of $\leb^2([t_0,T])$, one gets
\[ \sum_{j=1}^\infty \left(\int_{t_0}^t \phi_j(s)\,\dif s\right)^2=\sum_{j=1}^\infty \langle \mathbbm{1}_{[t_0,t]},\phi_j \rangle_{\leb^2([t_0,T])}^2=\|\mathbbm{1}_{[t_0,t]}\|_{\leb^2([t_0,T])}^2= t-t_0\leq T-t_0<\infty. \]
By Cauchy-Schwarz inequality for series,
\begin{align*}
 \sup_{t\in[t_0,T]}\sum_{j=1}^\infty \sqrt{\nu_j}\,\left|\int_{t_0}^t \phi_j(s)\,\dif s\right|\leq {} & \sup_{t\in[t_0,T]}\left(\sum_{j=1}^\infty \nu_j\right)^{\frac12}\left(\sum_{j=1}^\infty \left(\int_{t_0}^t \phi_j(s)\,\dif s\right)^2\right)^{\frac12} \\
\leq {} & \left(\sum_{j=1}^\infty \nu_j\right)^{\frac12}\sqrt{T-t_0}<\infty. 
\end{align*}
This finishes the proof of the remark.
\end{remark}

\begin{theorem} \label{teor3}
Assume that
\begin{align*}
 \text{H1}: {} & \;a,b\in \leb^2([t_0,T]\times\Omega),\;x_0\in\leb^2(\Omega); \\
 \text{H2}: {} & \;x_0,\,\eta_1,\,(\xi_1,\ldots,\xi_N,\eta_2,\ldots,\eta_N) \text{ are absolutely continuous and independent, }N\geq1; \\
 \text{H3}: {} & \;\text{the density function of }\eta_1\text{, }f_{\eta_1} \text{, is Lipschitz on }\mathbb{R}; \\
 \text{H4}: {} & \;\xi_1,\xi_2,\ldots \text{ have compact support in }[-A,A]\, (A>0) \text{ and }\psi_1>0\text{ on }(t_0,T). 
\end{align*}
Then, for each fixed $t\in (t_0,T]$, the sequence $\{f_1^N(x,t)\}_{N=1}^\infty$ given in (\ref{f1nou}) (which is the same as (\ref{f1n})) converges in $\leb^\infty(J)$ for every bounded set $J\subseteq\mathbb{R}$, to a density $f_1(x,t)$ of the solution (\ref{sol_ale}).
\end{theorem}
\begin{proof}
The idea is to prove that, for each fixed $t\in (t_0,T]$, the sequence $\{f_1^N(x,t)\}_{N=1}^\infty$ given in (\ref{f1nou}) is Cauchy in $\leb^\infty(J)$, for every bounded set $J\subseteq\mathbb{R}$.

First, we deal with some inequalities that will facilitate things later on. Fix two indexes $N>M$ and $t\in (t_0,T]$. Fix real numbers $\xi_1,\ldots,\xi_N$ that belong to $[-A,A]$ and real numbers $\eta_1,\ldots,\eta_N$. To make the notation easier, hereinafter we will denote by $C$ any constant independent of $N$ and $x$.

We have
\begin{align*}
& \left|\frac{\e^{-K_a(t,\pmb{\xi}_N)}}{\sqrt{\gamma_1}\int_{t_0}^t \psi_1(s)\e^{-K_a(s,\pmb{\xi}_N)}\,\dif s}-\frac{\e^{-K_a(t,\pmb{\xi}_M)}}{\sqrt{\gamma_1}\int_{t_0}^t \psi_1(s)\e^{-K_a(s,\pmb{\xi}_M)}\,\dif s}\right| \\
= {} & \left|\frac{\e^{-K_a(t,\pmb{\xi}_N)}\int_{t_0}^t \psi_1(s)\e^{-K_a(s,\pmb{\xi}_M)}\,\dif s-\e^{-K_a(t,\pmb{\xi}_M)}\int_{t_0}^t \psi_1(s)\e^{-K_a(s,\pmb{\xi}_N)}\,\dif s}{\sqrt{\gamma_1} \left(\int_{t_0}^t \psi_1(s)\e^{-K_a(s,\pmb{\xi}_N)}\,\dif s\right)\left(\int_{t_0}^t \psi_1(s)\e^{-K_a(s,\pmb{\xi}_M)}\,\dif s\right)}\right|.
\end{align*}
Recall that Remark \ref{rmk} and the boundedness of $\xi_1,\ldots,\xi_N$ imply (\ref{assum}), even more:
\begin{align}
 \int_{t_0}^t \psi_1(s)\e^{-K_a(s,\pmb{\xi}_M)}\,\dif s\geq {} & \int_{t_0}^t \psi_1(s)\e^{-\int_{t_0}^s \mu_a(r)\,\dif r-\sum_{k=1}^\infty \sqrt{\nu_k}\,\left|\int_{t_0}^t \phi_k(s)\,\dif s\right|A}\,\dif s \nonumber \\
\geq {} & \int_{t_0}^t \psi_1(s)\e^{-\int_{t_0}^s |\mu_a(r)|\,\dif r-C\,A}\,\dif s \\
\geq {} & \int_{t_0}^t \psi_1(s)\,\dif s\,\e^{-\int_{t_0}^T |\mu_a(r)|\,\dif r-C\,A} \\
\geq {} & \e^{-\|\mu_a\|_{\leb^2([t_0,T])}-C\,A}\|\psi_1\|_{\leb^1([t_0,t])}>0. \label{infe}
\end{align}
Then
\begin{align*}
& \left|\frac{\e^{-K_a(t,\pmb{\xi}_N)}}{\sqrt{\gamma_1}\int_{t_0}^t \psi_1(s)\e^{-K_a(s,\pmb{\xi}_N)}\,\dif s}-\frac{\e^{-K_a(t,\pmb{\xi}_M)}}{\sqrt{\gamma_1}\int_{t_0}^t \psi_1(s)\e^{-K_a(s,\pmb{\xi}_M)}\,\dif s}\right| \\
\leq {} & C\,\left|\e^{-K_a(t,\pmb{\xi}_N)}\int_{t_0}^t \psi_1(s)\e^{-K_a(s,\pmb{\xi}_M)}\,\dif s-\e^{-K_a(t,\pmb{\xi}_M)}\int_{t_0}^t \psi_1(s)\e^{-K_a(s,\pmb{\xi}_N)}\,\dif s\right| \\
\leq {} & C\bigg\{ \left(\int_{t_0}^t \psi_1(s)\e^{-K_a(s,\pmb{\xi}_M)}\,\dif s\right)\left|\e^{-K_a(t,\pmb{\xi}_N)}-\e^{-K_a(t,\pmb{\xi}_M)}\right| \\
+ {} & \e^{-K_a(t,\pmb{\xi}_M)}\left| \int_{t_0}^t \psi_1(s)\e^{-K_a(s,\pmb{\xi}_N)}\,\dif s-\int_{t_0}^t \psi_1(s)\e^{-K_a(s,\pmb{\xi}_M)}\,\dif s\right|\bigg\}.
\end{align*}
We have 
\begin{equation}
 e^{\pm K_a(t,\pmb{\xi}_N)}\leq \e^{\int_{t_0}^t|\mu_a(s)|\,\dif s+A\sum_{j=1}^\infty \sqrt{\nu_j}\,\left|\int_{t_0}^t \phi_j(s)\,\dif s\right|}\leq \e^{\|\mu_a\|_{\leb^2([t_0,T])}+C\,A},
\label{eka}
\end{equation}
whence
\begin{equation}
 \int_{t_0}^t \psi_1(s)\e^{-K_a(s,\pmb{\xi}_M)}\,\dif s\leq \e^{\|\mu_a\|_{\leb^2([t_0,T])}+C\,A}\|\psi_1\|_{\leb^1([t_0,T])}<\infty. 
 \label{ipsi}
\end{equation}
By (\ref{tvm}) (recall it was a consequence of the Mean Value Theorem) and (\ref{eka}),
\begin{align}
 \left|\e^{-K_a(t,\pmb{\xi}_N)}-\e^{-K_a(t,\pmb{\xi}_M)}\right|\leq {} & (\e^{-K_a(t,\pmb{\xi}_N)}+\e^{K_a(t,\pmb{\xi}_M)})|K_a(t,\pmb{\xi}_N)-K_a(t,\pmb{\xi}_M)| \nonumber \\
\leq {} & C|K_a(t,\pmb{\xi}_N)-K_a(t,\pmb{\xi}_M)|. \label{anm}
\end{align}
Finally, by (\ref{anm}) and Cauchy-Schwarz inequality for integrals,
\begin{align}
& \left| \int_{t_0}^t \psi_1(s) \e^{-K_a(s,\pmb{\xi}_N)}\,\dif s-\int_{t_0}^t \psi_1(s) \e^{-K_a(s,\pmb{\xi}_M)}\,\dif s\right| \nonumber \\
= {} & \left| \int_{t_0}^t \psi_1(s)\left( \e^{-K_a(s,\pmb{\xi}_N)}- \e^{-K_a(s,\pmb{\xi}_M)}\right)\,\dif s\right| \nonumber \\
\leq {} & \int_{t_0}^t \psi_1(s)\left|\e^{-K_a(s,\pmb{\xi}_N)}-\e^{-K_a(s,\pmb{\xi}_M)}\right|\,\dif s \nonumber \\
\leq {} & \int_{t_0}^T \psi_1(s)\left|\e^{-K_a(s,\pmb{\xi}_N)}-\e^{-K_a(s,\pmb{\xi}_M)}\right|\,\dif s \nonumber \\
\leq {} & \|\psi_1\|_{\leb^2([t_0,T])}\|\e^{-K_a(\cdot,\pmb{\xi}_N)}-\e^{-K_a(\cdot,\pmb{\xi}_M)}\|_{\leb^2([t_0,T])} \\
\leq {} & C \|K_a(\cdot,\pmb{\xi}_N)-K_a(\cdot,\pmb{\xi}_M)\|_{\leb^2([t_0,T])}. \label{vaig}
\end{align}
Inequalities (\ref{ipsi}), (\ref{anm}), (\ref{eka}) and (\ref{vaig}) yield
\begin{align}
 &\left|\frac{\e^{-K_a(t,\pmb{\xi}_N)}}{\sqrt{\gamma_1}\int_{t_0}^t \psi_1(s)\e^{-K_a(s,\pmb{\xi}_N)}\,\dif s}-\frac{\e^{-K_a(t,\pmb{\xi}_M)}}{\sqrt{\gamma_1}\int_{t_0}^t \psi_1(s)\e^{-K_a(s,\pmb{\xi}_M)}\,\dif s}\right| \nonumber \\
\leq {} & C\left\{ |K_a(t,\pmb{\xi}_N)-K_a(t,\pmb{\xi}_M)|+\| K_a(\cdot,\pmb{\xi}_N)-K_a(\cdot,\pmb{\xi}_M)\|_{\leb^2([t_0,T])}\right\} . \label{terme1}
\end{align}

Another bound that will be used later on, and which is a consequence of (\ref{infe}) and (\ref{eka}), is
\begin{equation}
 \frac{\e^{-K_a(t,\pmb{\xi}_M)}}{\sqrt{\gamma_1}\int_{t_0}^t \psi_1(s)\e^{-K_a(s,\pmb{\xi}_M)}\,\dif s}\leq C. 
\label{terme2}
\end{equation}

Let $L$ be the Lipschitz constant of $f_{\eta_1}$ on $\mathbb{R}$. Then
\footnotesize
\begin{align}
& \bigg|f_{\eta_1}\left(\frac{x\e^{-K_a(t,\pmb{\xi}_N)}-x_0-\int_{t_0}^t\left(\mu_b(s)+\sum_{i=2}^N\sqrt{\gamma_i}\,\psi_i(s)\eta_i\right)\e^{-K_a(s,\pmb{\xi}_N)}\,\dif s}{\sqrt{\gamma_1}\int_{t_0}^t \psi_1(s)\e^{-K_a(s,\pmb{\xi}_N)}\,\dif s}\right) \nonumber \\
- {} & f_{\eta_1}\left(\frac{x\e^{-K_a(t,\pmb{\xi}_M)}-x_0-\int_{t_0}^t\left(\mu_b(s)+\sum_{i=2}^M\sqrt{\gamma_i}\,\psi_i(s)\eta_i\right)\e^{-K_a(s,\pmb{\xi}_M)}\,\dif s}{\sqrt{\gamma_1}\int_{t_0}^t \psi_1(s)\e^{-K_a(s,\pmb{\xi}_M)}\,\dif s}\right)\bigg| \nonumber \\
\leq {} & L\,\bigg| \frac{x\e^{-K_a(t,\pmb{\xi}_N)}-x_0-\int_{t_0}^t\left(\mu_b(s)+\sum_{i=2}^N\sqrt{\gamma_i}\,\psi_i(s)\eta_i\right)\e^{-K_a(s,\pmb{\xi}_N)}\,\dif s}{\sqrt{\gamma_1}\int_{t_0}^t \psi_1(s)\e^{-K_a(s,\pmb{\xi}_N)}\,\dif s} \nonumber \\
- {} & \frac{x\e^{-K_a(t,\pmb{\xi}_M)}-x_0-\int_{t_0}^t\left(\mu_b(s)+\sum_{i=2}^M\sqrt{\gamma_i}\,\psi_i(s)\eta_i\right)\e^{-K_a(s,\pmb{\xi}_M)}\,\dif s}{\sqrt{\gamma_1}\int_{t_0}^t \psi_1(s)\e^{-K_a(s,\pmb{\xi}_M)}\,\dif s}\bigg| \nonumber \\
{} & \text{(common denominator to subtract both fractions and inequality (\ref{infe}))} \nonumber \\
\leq {} & C\,\bigg|\left(\int_{t_0}^t \psi_1(s)\e^{-K_a(s,\pmb{\xi}_M)}\,\dif s\right)\left(x\e^{-K_a(t,\pmb{\xi}_N)}-x_0-\int_{t_0}^t\left(\mu_b(s)+\sum_{i=2}^N\sqrt{\gamma_i}\,\psi_i(s)\eta_i\right)\e^{-K_a(s,\pmb{\xi}_N)}\,\dif s\right) \nonumber \\
- {} & \left(\int_{t_0}^t \psi_1(s)\e^{-K_a(s,\pmb{\xi}_N)}\,\dif s\right)\left(x\e^{-K_a(t,\pmb{\xi}_M)}-x_0-\int_{t_0}^t\left(\mu_b(s)+\sum_{i=2}^M\sqrt{\gamma_i}\,\psi_i(s)\eta_i\right)\e^{-K_a(s,\pmb{\xi}_M)}\,\dif s\right)\bigg| \nonumber \\
{} & \text{(add and subtract, use bounds (\ref{ipsi}) and (\ref{vaig}) and triangular inequality)} \nonumber \\
\leq {} & C\bigg\{ \| K_a(\cdot,\pmb{\xi}_N)-K_a(\cdot,\pmb{\xi}_M)\|_{\leb^2([t_0,T])}\left(|x|\e^{-K_a(t,\pmb{\xi}_N)}+|x_0|+\int_{t_0}^t\left|\mu_b(s)+\sum_{i=2}^M\sqrt{\gamma_i}\,\psi_i(s)\eta_i\right|\e^{-K_a(s,\pmb{\xi}_M)}\,\dif s\right) \nonumber \\
+ {} & |x|\left|\e^{-K_a(t,\pmb{\xi}_N)}-e^{-K_a(t,\pmb{\xi}_M)}\right| \nonumber \\
+ {} & \left| \int_{t_0}^t \left(\mu_b(s)+\sum_{i=2}^M\sqrt{\gamma_i}\,\psi_i(s)\eta_i\right)\e^{-K_a(s,\pmb{\xi}_M)}\,\dif s - \int_{t_0}^t \left(\mu_b(s)+\sum_{i=2}^N\sqrt{\gamma_i}\,\psi_i(s)\eta_i\right)\e^{-K_a(s,\pmb{\xi}_N)}\,\dif s\right|\bigg\} \nonumber \\
{} & \text{(bounds (\ref{eka}) and (\ref{anm}))} \nonumber \\
\leq {} & C\left\{ |x|+|x_0|+\int_{t_0}^t\left|\mu_b(s)+\sum_{i=2}^M\sqrt{\gamma_i}\,\psi_i(s)\eta_i\right|\,\dif s\right\}\| K_a(\cdot,\pmb{\xi}_N)-K_a(\cdot,\pmb{\xi}_M)\|_{\leb^2([t_0,T])} \nonumber \\
+ {} & C\,|x||K_a(t,\pmb{\xi}_N)-K_a(t,\pmb{\xi}_M)| \nonumber \\
+ {} & C\int_{t_0}^t \left|\e^{-K_a(s,\pmb{\xi}_N)}-\e^{-K_a(s,\pmb{\xi}_M)}\right|\left|\mu_b(s)+\sum_{i=2}^N\sqrt{\gamma_i}\,\psi_i(s)\eta_i\right|\,\dif s+ \int_{t_0}^t \e^{-K_a(s,\pmb{\xi}_N)}\left|\sum_{i=M+1}^N\sqrt{\gamma_i}\,\psi_i(s)\eta_i\right|\,\dif s \nonumber \\
{} & \text{(Cauchy-Schwarz, (\ref{eka}) and (\ref{anm}))} \nonumber \\
\leq {} & C\left\{ |x|+|x_0|+\left\|\mu_b(\cdot)+\sum_{i=2}^M\sqrt{\gamma_i}\,\psi_i(\cdot)\eta_i\right\|_{\leb^2([t_0,T])}\right\}\| K_a(\cdot,\pmb{\xi}_N)-K_a(\cdot,\pmb{\xi}_M)\|_{\leb^2([t_0,T])} \nonumber \\
+ {} & C\,|x||K_a(t,\pmb{\xi}_N)-K_a(t,\pmb{\xi}_M)| \nonumber \\
+ {} & C\|K_a(\cdot,\pmb{\xi}_N)-K_a(\cdot,\pmb{\xi}_M)\|_{\leb^2([t_0,T])}\left\|\mu_b(\cdot)+\sum_{i=2}^N\sqrt{\gamma_i}\,\psi_i(\cdot)\eta_i\right\|_{\leb^2([t_0,T])}+ C\left\|\sum_{i=M+1}^N\sqrt{\gamma_i}\,\psi_i(\cdot)\eta_i\right\|_{\leb^2([t_0,T])}. \label{terme3} 
\end{align}
\normalsize

Since $f_{\eta_1}$ is Lipschitz and integrable on $\mathbb{R}$, it is bounded, therefore
\footnotesize
\begin{equation}
 f_{\eta_1}\left(\frac{x\e^{-K_a(t,\pmb{\xi}_N)}-x_0-\int_{t_0}^t\left(\mu_b(s)+\sum_{i=2}^N\sqrt{\gamma_i}\,\psi_i(s)\eta_i\right)\e^{-K_a(s,\pmb{\xi}_N)}\,\dif s}{\sqrt{\gamma_1}\int_{t_0}^t \psi_1(s)\e^{-K_a(s,\pmb{\xi}_N)}\,\dif s}\right) \leq C. 
\label{terme4}
\end{equation}
\normalsize

We estimate $|f_1^N(x,t)-f_1^M(x,t)|$. Using expression (\ref{f1nou}) and taking marginal distributions with respect to $\xi_{M+1},\ldots,\xi_N,\eta_{M+1},\ldots,\eta_N$, we obtain the following expression for $f_1^M(x,t)$:
\footnotesize
\begin{align*}
 f_1^{M}(x,t)= {} & \int_{\mathcal{D}_M} f_{\eta_1}\left(\frac{x\e^{-K_a(t,\pmb{\xi}_M)}-x_0-\int_{t_0}^t\left(\mu_b(s)+\sum_{i=2}^M\sqrt{\gamma_i}\,\psi_i(s)\eta_i\right)\e^{-K_a(s,\pmb{\xi}_M)}\,\dif s}{\sqrt{\gamma_1}\int_{t_0}^t \psi_1(s)\e^{-K_a(s,\pmb{\xi}_M)}\,\dif s}\right) \nonumber \\
\cdot {} & f_0(x_0)f_{(\xi_1,\ldots,\xi_M,\eta_2,\ldots,\eta_M)}(\xi_1,\ldots,\xi_M,\eta_2,\ldots,\eta_M) \nonumber \\
\cdot {} & \frac{\e^{-K_a(t,\pmb{\xi}_M)}}{\sqrt{\gamma_1}\int_{t_0}^t \psi_1(s)\e^{-K_a(s,\pmb{\xi}_M)}\,\dif s} \,\dif x_0\,\dif\xi_1\cdots \dif\xi_M\,\dif\eta_2\cdots \dif\eta_M \\
= {} & \int_{\mathcal{D}_N} f_{\eta_1}\left(\frac{x\e^{-K_a(t,\pmb{\xi}_M)}-x_0-\int_{t_0}^t\left(\mu_b(s)+\sum_{i=2}^M\sqrt{\gamma_i}\,\psi_i(s)\eta_i\right)\e^{-K_a(s,\pmb{\xi}_M)}\,\dif s}{\sqrt{\gamma_1}\int_{t_0}^t \psi_1(s)\e^{-K_a(s,\pmb{\xi}_M)}\,\dif s}\right) \nonumber \\
\cdot {} & f_0(x_0)f_{(\xi_1,\ldots,\xi_N,\eta_2,\ldots,\eta_N)}(\xi_1,\ldots,\xi_N,\eta_2,\ldots,\eta_N) \nonumber \\
\cdot {} & \frac{\e^{-K_a(t,\pmb{\xi}_M)}}{\sqrt{\gamma_1}\int_{t_0}^t \psi_1(s)\e^{-K_a(s,\pmb{\xi}_M)}\,\dif s} \,\dif x_0\,\dif\xi_1\cdots \dif\xi_N\,\dif\eta_2\cdots \dif\eta_N.
\end{align*}
\normalsize
Then, using the triangular inequality,
\footnotesize
\begin{align*}
{} & |f_1^N(x,t)-f_1^M(x,t)| \\
\leq {} & \int_{\mathcal{D}_N} \bigg|f_{\eta_1}\left(\frac{x\e^{-K_a(t,\pmb{\xi}_N)}-x_0-\int_{t_0}^t\left(\mu_b(s)+\sum_{i=2}^N\sqrt{\gamma_i}\,\psi_i(s)\eta_i\right)\e^{-K_a(s,\pmb{\xi}_N)}\,\dif s}{\sqrt{\gamma_1}\int_{t_0}^t \psi_1(s)\e^{-K_a(s,\pmb{\xi}_N)}\,\dif s}\right)\frac{\e^{-K_a(t,\pmb{\xi}_N)}}{\sqrt{\gamma_1}\int_{t_0}^t \psi_1(s)\e^{-K_a(s,\pmb{\xi}_N)}\,\dif s}  \\
- {} & f_{\eta_1}\left(\frac{x\e^{-K_a(t,\pmb{\xi}_M)}-x_0-\int_{t_0}^t\left(\mu_b(s)+\sum_{i=2}^M\sqrt{\gamma_i}\,\psi_i(s)\eta_i\right)\e^{-K_a(s,\pmb{\xi}_M)}\,\dif s}{\sqrt{\gamma_1}\int_{t_0}^t \psi_1(s)\e^{-K_a(s,\pmb{\xi}_M)}\,\dif s}\right)\frac{\e^{-K_a(t,\pmb{\xi}_M)}}{\sqrt{\gamma_1}\int_{t_0}^t \psi_1(s)\e^{-K_a(s,\pmb{\xi}_M)}\,\dif s}\bigg| \\
\cdot {} & f_0(x_0)f_{(\xi_1,\ldots,\xi_N,\eta_2,\ldots,\eta_N)}(\xi_1,\ldots,\xi_N,\eta_2,\ldots,\eta_N)\,\dif x_0\,\dif\xi_1\cdots \dif\xi_N\,\dif\eta_2\cdots \dif\eta_N \\
\leq {} & \int_{\mathcal{D}_N} \bigg|f_{\eta_1}\left(\frac{x\e^{-K_a(t,\pmb{\xi}_N)}-x_0-\int_{t_0}^t\left(\mu_b(s)+\sum_{i=2}^N\sqrt{\gamma_i}\,\psi_i(s)\eta_i\right)\e^{-K_a(s,\pmb{\xi}_N)}\,\dif s}{\sqrt{\gamma_1}\int_{t_0}^t \psi_1(s)\e^{-K_a(s,\pmb{\xi}_N)}\,\dif s}\right) \\
- {} & f_{\eta_1}\left(\frac{x\e^{-K_a(t,\pmb{\xi}_M)}-x_0-\int_{t_0}^t\left(\mu_b(s)+\sum_{i=2}^M\sqrt{\gamma_i}\,\psi_i(s)\eta_i\right)\e^{-K_a(s,\pmb{\xi}_M)}\,\dif s}{\sqrt{\gamma_1}\int_{t_0}^t \psi_1(s)\e^{-K_a(s,\pmb{\xi}_M)}\,\dif s}\right)\bigg| \\
\cdot {} & \frac{\e^{-K_a(t,\pmb{\xi}_M)}}{\sqrt{\gamma_1}\int_{t_0}^t \psi_1(s)\e^{-K_a(s,\pmb{\xi}_M)}\,\dif s} f_0(x_0)f_{(\xi_1,\ldots,\xi_N,\eta_2,\ldots,\eta_N)}(\xi_1,\ldots,\xi_N,\eta_2,\ldots,\eta_N)\,\dif x_0\,\dif\xi_1\cdots \dif\xi_N\,\dif\eta_2\cdots \dif\eta_N \\
+ {} & \int_{\mathcal{D}_N} f_{\eta_1}\left(\frac{x\e^{-K_a(t,\pmb{\xi}_N)}-x_0-\int_{t_0}^t\left(\mu_b(s)+\sum_{i=2}^N\sqrt{\gamma_i}\,\psi_i(s)\eta_i\right)\e^{-K_a(s,\pmb{\xi}_N)}\,\dif s}{\sqrt{\gamma_1}\int_{t_0}^t \psi_1(s)\e^{-K_a(s,\pmb{\xi}_N)}\,\dif s}\right) \\
\cdot {} & \left| \frac{\e^{-K_a(t,\pmb{\xi}_N)}}{\sqrt{\gamma_1}\int_{t_0}^t \psi_1(s)\e^{-K_a(s,\pmb{\xi}_N)}\,\dif s} - \frac{\e^{-K_a(t,\pmb{\xi}_M)}}{\sqrt{\gamma_1}\int_{t_0}^t \psi_1(s)\e^{-K_a(s,\pmb{\xi}_M)}\,\dif s} \right| \\
\cdot {} & f_0(x_0)f_{(\xi_1,\ldots,\xi_N,\eta_2,\ldots,\eta_N)}(\xi_1,\ldots,\xi_N,\eta_2,\ldots,\eta_N)\,\dif x_0\,\dif\xi_1\cdots \dif\xi_N\,\dif\eta_2\cdots \dif\eta_N.
\end{align*}
\normalsize
Using the definition of expectation as an integral with respect to the corresponding density function, bounds (\ref{terme1}), (\ref{terme2}), (\ref{terme3}) and (\ref{terme4}), Cauchy-Schwarz inequality and (\ref{ka}), we have:
\footnotesize
\begin{align*}
& |f_1^N(x,t)-f_1^M(x,t)| \\
\leq {} & C\bigg\{ \mathbb{E}\bigg[|K_a(t,\pmb{\xi}_N)-K_a(t,\pmb{\xi}_M)|+\|K_a(\cdot,\pmb{\xi}_N)-K_a(\cdot,\pmb{\xi}_M)\|_{\leb^2([t_0,T])}\bigg] \\
+ {} & \mathbb{E}\left[\left(|x|+|x_0|+\left\|\mu_b(\cdot)+\sum_{i=2}^M\sqrt{\gamma_i}\,\psi_i(\cdot)\eta_i\right\|_{\leb^2([t_0,T])}\right)\|K_a(\cdot,\pmb{\xi}_N)-K_a(\cdot,\pmb{\xi}_M)\|_{\leb^2([t_0,T])}\right] \\
+ {} & |x|\,\mathbb{E}\left[|K_a(t,\pmb{\xi}_N)-K_a(t,\pmb{\xi}_M)|\right] \\
+ {} & \mathbb{E}\left[\|K_a(\cdot,\pmb{\xi}_N)-K_a(\cdot,\pmb{\xi}_M)\|_{\leb^2([t_0,T])}\left\|\mu_b(\cdot)+\sum_{i=2}^N\sqrt{\gamma_i}\,\psi_i(\cdot)\eta_i\right\|_{\leb^2([t_0,T])} \right] \\
+ {} & \mathbb{E}\left[\left\|\sum_{i=M+1}^N\sqrt{\gamma_i}\,\psi_i(\cdot)\eta_i\right\|_{\leb^2([t_0,T])}\right] \bigg\} \\
\leq {} & C\bigg\{(|x|+1)\mathbb{E}[|K_a(t,\pmb{\xi}_N)-K_a(t,\pmb{\xi}_M)|]+(|x|+\|x_0\|_{\leb^2(\Omega)}+1)\mathbb{E}\left[\|K_a(\cdot,\pmb{\xi}_N)-K_a(\cdot,\pmb{\xi}_M)\|_{\leb^2([t_0,T])}^2\right]^\frac12 \\
+ {} & \mathbb{E}\left[\|K_a(\cdot,\pmb{\xi}_N)-K_a(\cdot,\pmb{\xi}_M)\|_{\leb^2([t_0,T])}\left\|\mu_b(\cdot)+\sum_{i=2}^M\sqrt{\gamma_i}\,\psi_i(\cdot)\eta_i\right\|_{\leb^2([t_0,T])}\right] \\
+ {} & \mathbb{E}\left[\|K_a(\cdot,\pmb{\xi}_N)-K_a(\cdot,\pmb{\xi}_M)\|_{\leb^2([t_0,T])}\left\|\mu_b(\cdot)+\sum_{i=2}^N\sqrt{\gamma_i}\,\psi_i(\cdot)\eta_i\right\|_{\leb^2([t_0,T])}\right] \\
+ {} & \mathbb{E}\left[\left\|\sum_{i=M+1}^N\sqrt{\gamma_i}\,\psi_i(\cdot)\eta_i\right\|_{\leb^2([t_0,T])}\right]\bigg\} \\
\leq {} & C\bigg\{(|x|+1)\|K_a(t,\pmb{\xi}_N)-K_a(t,\pmb{\xi}_M)\|_{\leb^2(\Omega)}+(|x|+\|x_0\|_{\leb^2(\Omega)}+1)\mathbb{E}\left[\|K_a(\cdot,\pmb{\xi}_N)-K_a(\cdot,\pmb{\xi}_M)\|_{\leb^2([t_0,T])}^2\right]^{\frac12} \\
+ {} & \mathbb{E}\left[\|K_a(\cdot,\pmb{\xi}_N)-K_a(\cdot,\pmb{\xi}_M)\|_{\leb^2([t_0,T])}^2\right]^{\frac12} \mathbb{E}\left[\left\|\mu_b(\cdot)+\sum_{i=2}^M\sqrt{\gamma_i}\,\psi_i(\cdot)\eta_i\right\|_{\leb^2([t_0,T])}^2\right]^{\frac12} \\
+ {} & \mathbb{E}\left[\|K_a(\cdot,\pmb{\xi}_N)-K_a(\cdot,\pmb{\xi}_M)\|_{\leb^2([t_0,T])}^2\right]^{\frac12} \mathbb{E}\left[\left\|\mu_b(\cdot)+\sum_{i=2}^N\sqrt{\gamma_i}\,\psi_i(\cdot)\eta_i\right\|_{\leb^2([t_0,T])}^2\right]^{\frac12} \\
+ {} & \mathbb{E}\left[\left\|\sum_{i=M+1}^N\sqrt{\gamma_i}\,\psi_i(\cdot)\eta_i\right\|_{\leb^2([t_0,T])}^2\right]^{\frac12} \bigg\} \\
\leq {} & C\bigg\{\left(|x|+\|x_0\|_{\leb^2(\Omega)}+1+\left\|\mu_b(\cdot)+\sum_{i=2}^M\sqrt{\gamma_i}\,\psi_i(\cdot)\eta_i\right\|_{\leb^2([t_0,T]\times\Omega)}+\left\|\mu_b(\cdot)+\sum_{i=2}^N\sqrt{\gamma_i}\,\psi_i(\cdot)\eta_i\right\|_{\leb^2([t_0,T]\times\Omega)}\right) \\
\cdot  {} & \|a_N-a_M\|_{\leb^2([t_0,T]\times\Omega)}+\left\|\sum_{i=M+1}^N\sqrt{\gamma_i}\,\psi_i(\cdot)\eta_i\right\|_{\leb^2([t_0,T]\times\Omega)} \bigg\}.
\end{align*}
\normalsize
Since $a_N\rightarrow a$ and $b_N\rightarrow b$ as $N\rightarrow\infty$ in $\leb^2([t_0,T]\times\Omega)$, we conclude that the sequence $\{f_1^N(x,t)\}_{N=1}^\infty$ given in (\ref{f1nou}) is Cauchy in $\leb^\infty(J)$, for every bounded set $J\subseteq\mathbb{R}$. 

As we saw in the proof of Theorem \ref{teor1}, $x_{N,N}(t,\omega)\rightarrow x(t,\omega)$ as $N\rightarrow\infty$ for all $t\in [t_0,T]$ and a.e. $\omega\in\Omega$. As we showed there, this fact is enough to ensure that the limit of the sequence $\{f_1^N(x,t)\}_{N=1}^\infty$ is a density of the process $x(t,\omega)$ given in (\ref{sol_ale}).

\end{proof}

\subsection{Obtaining the density function when $b=0$ and \texorpdfstring{$f_{\xi_1}$}{fxi1} is Lipchitz on $\mathbb{R}$} \label{su4} \ \\

If $b=0$, then the truncation (\ref{xNM}) becomes (\ref{truncamm}),
\[ x_N(t,\omega)=x_0(\omega)\e^{K_a(t,\pmb{\xi}_N(\omega))}. \]
We use Lemma \ref{lema_abscont} to compute (\ref{f1Nhomo}) in a different way. The idea is that, instead of isolating $x_0$, we isolate $\xi_1$. Indeed, in the notation of Lemma \ref{lema_abscont},
\[ g(\xi_1,\ldots,\xi_N,x_0)=\left(x_0\e^{K_a(t,\pmb{\xi}_N)},\xi_2,\ldots,\xi_N,x_0\right), \]
\[ D=\mathbb{R}^N\times\{x_0\in\mathbb{R}:\,x_0\neq0\}, \] 
\begin{align*}
 g(D)= {} & \{(\xi_1,\ldots,\xi_N,x_0)\in\mathbb{R}^{N+1}:\,\xi_1/x_0>0,\,x_0\neq0\} \\
= {} & ((0,\infty)\times\mathbb{R}^{N-2}\times (0,\infty))\cup ((-\infty,0)\times \mathbb{R}^{N-2}\times (-\infty,0)), 
\end{align*}
\footnotesize
\begin{align*}
{} & h(\xi_1,\xi_2,\ldots,\xi_N,x_0) \\
= {} & \left(\frac{1}{\sqrt{\nu_1}\,\int_{t_0}^t\phi_1(s)\,\dif s} \left\{\log\left(\frac{\xi_1}{x_0}\right)-\int_{t_0}^t\mu_a(s)\,\dif s-\sum_{j=2}^N\sqrt{\nu_j}\,\left(\int_{t_0}^t\phi_j(s)\,\dif s\right)\xi_j\right\},\xi_2,\ldots,\xi_N,x_0\right) 
\end{align*}
\normalsize
and, assuming that $\int_{t_0}^t\phi_1(s)\,\dif s\neq0$ for all $t\in (t_0,T]$, 
\[ |Jh(\xi_1,\ldots,\xi_N,x_0)|=\frac{1}{|\xi_1|\sqrt{\nu_1}\,\left|\int_{t_0}^t \phi_1(s)\,\dif s\right|}\neq0. \]
Assume independence of $x_0$, $\xi_1$ and $(\xi_2,\ldots,\xi_N)$, $N\geq2$. Then, taking the marginal distributions we arrive at the following form of (\ref{f1Nhomo}),
\footnotesize
\begin{align}
 {} & f_1^N(x,t) \nonumber \\
= {} & \int_{\mathbb{R}^{N-1}\times I_{\mathrm{sign}(x)}} f_{\xi_1}\left(\frac{1}{\sqrt{\nu_1}\,\int_{t_0}^t\phi_1(s)\,\dif s} \left\{\log\left(\frac{x}{x_0}\right)-\int_{t_0}^t\mu_a(s)\,\dif s-\sum_{j=2}^N\sqrt{\nu_j}\,\left(\int_{t_0}^t\phi_j(s)\,\dif s\right)\xi_j\right\}\right) \nonumber \\
\cdot {} & f_{(\xi_2,\ldots,\xi_N)}(\xi_2,\ldots,\xi_N)f_0(x_0)\frac{1}{|x|\sqrt{\nu_1}\,\left|\int_{t_0}^t \phi_1(s)\,\dif s\right|}\,\dif \xi_2\cdots\dif \xi_N\,\dif x_0, \label{f1ncall}
\end{align}
\normalsize
where $\mathrm{sign}(x)=+$ if $x>0$ and $\mathrm{sign}(x)=-$ if $x<0$, and $I_+=(0,\infty)$ and $I_-=(-\infty,0)$. This density function (\ref{f1ncall}) is not defined at $x=0$, but it does not matter since density functions may only be defined almost everywhere on $\mathbb{R}$.

\begin{theorem} \label{teorb0L}
Assume that:
\begin{align*}
 \text{H1}: {} & \;a\in \leb^2([t_0,T]\times\Omega); \\
 \text{H2}: {} & \;x_0,\;\xi_1\text{ and }(\xi_2,\ldots,\xi_N) \text{ are absolutely continuous and independent, }N\geq2; \\
 \text{H3}: {} & \;\text{the density function of }\xi_1\text{, }f_{\xi_1} \text{, is Lipschitz on }\mathbb{R}; \\
 \text{H4}: {} & \; \int_{t_0}^t\phi_1(s)\,\dif s\neq0\text{ for all }t\in (t_0,T].
\end{align*}
Then, for each fixed $t\in (t_0,T]$, the sequence $\{f_1^N(x,t)\}_{N=1}^\infty$ given in (\ref{f1ncall}) (which is the same as (\ref{f1Nhomo})) converges in $\leb^\infty(J)$ for every bounded set $J\subseteq\mathbb{R}\backslash [-\delta,\delta]$, for every $\delta>0$, to a density $f_1(x,t)$ of the solution stochastic process (\ref{xhomo}).
\end{theorem}

\begin{proof}
Let us see that, for each fixed $t$, $\{f_1^N(x,t)\}_{N=1}^\infty$ is Cauchy in $\leb^\infty(J)$, for every bounded set $J\subseteq\mathbb{R}\backslash [-\delta,\delta]$, $\delta>0$. Fix two indexes $N>M\geq2$. If we denote by $L$ the Lipschitz constant of $f_{\xi_1}$, we have:
\begin{align*}
{} & \bigg|f_{\xi_1}\left(\frac{1}{\sqrt{\nu_1}\,\int_{t_0}^t\phi_1(s)\,\dif s} \left\{\log\left(\frac{x}{x_0}\right)-\int_{t_0}^t\mu_a(s)\,\dif s-\sum_{j=2}^N\sqrt{\nu_j}\,\left(\int_{t_0}^t\phi_j(s)\,\dif s\right)\xi_j\right\}\right) \\
- {} & f_{\xi_1}\left(\frac{1}{\sqrt{\nu_1}\,\int_{t_0}^t\phi_1(s)\,\dif s} \left\{\log\left(\frac{x}{x_0}\right)-\int_{t_0}^t\mu_a(s)\,\dif s-\sum_{j=2}^M\sqrt{\nu_j}\,\left(\int_{t_0}^t\phi_j(s)\,\dif s\right)\xi_j\right\}\right)\bigg| \\
\leq {} & L\,\frac{1}{\sqrt{\nu_1}\,\left|\int_{t_0}^t\phi_1(s)\,\dif s\right|}\sum_{j=M+1}^N \sqrt{\nu_j}\,\left|\int_{t_0}^t\phi_j(s)\,\dif s\right||\xi_j|.
\end{align*}
Taking marginal distributions in expression (\ref{f1ncall}), one gets
\footnotesize
\begin{align*}
 {} & f_1^M(x,t) \nonumber \\
= {} & \int_{\mathbb{R}^{M-1}\times I_{\mathrm{sign}(x)}} f_{\xi_1}\left(\frac{1}{\sqrt{\nu_1}\,\int_{t_0}^t\phi_1(s)\,\dif s} \left\{\log\left(\frac{x}{x_0}\right)-\int_{t_0}^t\mu_a(s)\,\dif s-\sum_{j=2}^M\sqrt{\nu_j}\,\left(\int_{t_0}^t\phi_j(s)\,\dif s\right)\xi_j\right\}\right) \nonumber \\
\cdot {} & f_{(\xi_2,\ldots,\xi_M)}(\xi_2,\ldots,\xi_M)f_0(x_0)\frac{1}{|x|\sqrt{\nu_1}\,\left|\int_{t_0}^t \phi_1(s)\,\dif s\right|}\,\dif \xi_2\cdots\dif \xi_M\,\dif x_0 \\
= {} & \int_{\mathbb{R}^{N-1}\times I_{\mathrm{sign}(x)}} f_{\xi_1}\left(\frac{1}{\sqrt{\nu_1}\,\int_{t_0}^t\phi_1(s)\,\dif s} \left\{\log\left(\frac{x}{x_0}\right)-\int_{t_0}^t\mu_a(s)\,\dif s-\sum_{j=2}^M\sqrt{\nu_j}\,\left(\int_{t_0}^t\phi_j(s)\,\dif s\right)\xi_j\right\}\right) \nonumber \\
\cdot {} & f_{(\xi_2,\ldots,\xi_N)}(\xi_2,\ldots,\xi_N)f_0(x_0)\frac{1}{|x|\sqrt{\nu_1}\,\left|\int_{t_0}^t \phi_1(s)\,\dif s\right|}\,\dif \xi_2\cdots\dif \xi_N\,\dif x_0.
\end{align*}
\normalsize
We can estimate
\small
\begin{align*}
{} & |f_1^N(x,t)-f_1^M(x,t)| \\
\leq {} & L\int_{\mathbb{R}^{N-1}\times I_{\mathrm{sign}(x)}} \frac{1}{\sqrt{\nu_1}\,\left|\int_{t_0}^t\phi_1(s)\,\dif s\right|}\sum_{j=M+1}^N \left\{\sqrt{\nu_j}\,\left|\int_{t_0}^t\phi_j(s)\,\dif s\right||\xi_j|\right\} \\
\cdot {} & f_{(\xi_2,\ldots,\xi_N)}(\xi_2,\ldots,\xi_N)f_0(x_0)\frac{1}{|x|\sqrt{\nu_1}\,\left|\int_{t_0}^t \phi_1(s)\,\dif s\right|}\,\dif \xi_2\cdots\dif \xi_N\,\dif x_0 \\
= {} & L\int_{\mathbb{R}^{N-1}\times I_{\mathrm{sign}(x)}} \frac{1}{|x|\nu_1\,\left(\int_{t_0}^t\phi_1(s)\,\dif s\right)^2}\sum_{j=M+1}^N \left\{\sqrt{\nu_j}\,\left|\int_{t_0}^t\phi_j(s)\,\dif s\right||\xi_j|\right\} \\
\cdot {} & f_{(\xi_2,\ldots,\xi_N)}(\xi_2,\ldots,\xi_N)f_0(x_0)\,\dif \xi_2\cdots\dif \xi_N\,\dif x_0 \\
\leq {} & L\,\frac{1}{|x|\nu_1\,\left(\int_{t_0}^t\phi_1(s)\,\dif s\right)^2}\sum_{j=M+1}^N \sqrt{\nu_j}\,\left|\int_{t_0}^t\phi_j(s)\,\dif s\right|\mathbb{E}[|\xi_j|].
\end{align*}
\normalsize
Since $\sum_{j=1}^\infty \sqrt{\nu_j}\,\left|\int_{t_0}^t\phi_j(s)\,\dif s\right|\,\mathbb{E}[|\xi_j|]<\infty$ by Remark \ref{rmk}, we obtain that $\{f_1^N(x,t)\}_{N=1}^\infty$ is Cauchy in $\leb^\infty(J)$, for every bounded set $J\subseteq\mathbb{R}\backslash [-\delta,\delta]$, for every $\delta>0$.

As we saw in the end of Theorem \ref{teor1}, the truncation $x_N(t,\omega)$ given in (\ref{truncamm}) converges to the process $x(t,\omega)$ given in (\ref{xhomo}) for all $t$ and a.e. $\omega$ as $N\rightarrow\infty$. This allows us to conclude that the limit of $\{f_1^N(x,t)\}_{N=1}^\infty$ is the density function of the solution $x(t,\omega)$ given in (\ref{xhomo}).

\end{proof}

\subsection{Obtaining the density function under the weaker assumption of continuity} \label{su5} \ \\

We present some results that substitute the Lipschitz hypothesis by a continuity assumption. Notice that we ``only'' prove a pointwise convergence to the density of $x(t,\omega)$, not a uniform convergence on compact sets, as we did in the previous theorems.

\begin{theorem} \label{teornou1}
Assume the following four hypotheses:
\begin{align*}
 \text{H1}: {} & \;a,b\in \leb^2([t_0,T]\times\Omega); \\
 \text{H2}: {} & \;x_0 \text{ and }(\xi_1,\ldots,\xi_N,\eta_1,\ldots,\eta_N) \text{ are absolutely continuous and independent, }N\geq1; \\
 \text{H3}: {} & \;\text{the density function of }x_0\text{, }f_0 \text{, is continuous and bounded on }\mathbb{R}; \\ 
\text{H4}: {} & \;\|\e^{-K_a(t,\pmb{\xi}_N)}\|_{\leb^2(\Omega)}\leq C,\text{ for all }N\geq1\text{ and }t\in [t_0,T].
\end{align*}
Then, for all $x\in\mathbb{R}$ and $t\in [t_0,T]$, the sequence $\{f_1^N(x,t)\}_{N=1}^\infty$ given in (\ref{f1n}) converges to a density $f_1(x,t)$ of the solution $x(t,\omega)$ given in (\ref{sol_ale}).
\end{theorem} 
\begin{proof}
Fix $x\in\mathbb{R}$ and $t\in [t_0,T]$. If we define the random variables
\[ Y_N(\omega):=\e^{-K_a(t,\pmb{\xi}_N(\omega))},\quad Z_N(\omega):=\int_{t_0}^t S_b(s,\pmb{\eta}_N(\omega))\e^{-K_a(s,\pmb{\xi}_N(\omega))}\,\dif s, \]
then notice that the density function given by (\ref{f1n}) becomes
\begin{equation}
 f_1^N(x,t)=\mathbb{E}\left[f_0(xY_N-Z_N)Y_N\right]. 
 \label{f1nexp}
\end{equation}
By (\ref{lima}) and (\ref{limb}), we have
\[ \lim_{N\rightarrow\infty} \e^{-K_a(t,\pmb{\xi}_N(\omega))}=\e^{-\int_{t_0}^t a(s,\omega)\,\dif s}=:Y(\omega) \]
and
\[ \lim_{N\rightarrow\infty} \int_{t_0}^t S_b(s,\pmb{\eta}_N(\omega))\e^{-K_a(s,\pmb{\xi}_N(\omega))}\,\dif s=\int_{t_0}^t b(s,\omega)\e^{-\int_{t_0}^s a(r,\omega)\,\dif r}\,\dif s=:Z(\omega), \]
for a.e. $\omega\in\Omega$. Then $(Y_N(\omega),Z_N(\omega))\rightarrow (Y(\omega),Z(\omega))$ as $N\rightarrow\infty$, for a.e. $\omega\in\Omega$. 

Let $g(x,t)=\mathbb{E}[f_0(xY-Z)Y]$. By the triangular and Cauchy-Schwarz inequalities, we can estimate the difference between (\ref{f1nexp}) and $g(x,t)$:
\small
\begin{align*}
 |f_1^N(x,t)- {} & g(x,t)|\leq \mathbb{E}[|f_0(xY_N-Z_N)-f_0(xY-Z)||Y_N|]+\mathbb{E}[f_0(xY-Z)|Y_N-Y|] \\
\leq {} & \mathbb{E}[|f_0(xY_N- Z_N)-f_0(xY-Z)|^2]^{\frac12}\mathbb{E}[Y_N^2]^{\frac12}+\mathbb{E}[f_0(xY-Z)|Y_N-Y|].
\end{align*}
\normalsize
By hypotheses H3 and H4, 
\[ |f_1^N(x,t)-g(x,t)|\leq C\,\mathbb{E}[|f_0(xY_N- Z_N)-f_0(xY-Z)|^2]^{\frac12}+\|f_0\|_{\leb^\infty(\mathbb{R})} \mathbb{E}[|Y_N-Y|]. \]

As $f_0$ is continuous on $\mathbb{R}$, $|f_0(xY_N(\omega)- Z_N(\omega))-f_0(xY(\omega)-Z(\omega))|^2\rightarrow0$ as $N\rightarrow\infty$, for a.e. $\omega\in\Omega$. Since $f_0$ is bounded, by the Dominated Convergence Theorem \cite[result 11.32, p.321]{rudin},
\[ \lim_{N\rightarrow\infty} \mathbb{E}[|f_0(xY_N- Z_N)-f_0(xY-Z)|^2]^{\frac12}=0. \]

On the other hand, as a consequence of the Mean Value Theorem used in (\ref{tvm}),
\[ |Y_N(\omega)-Y_M(\omega)|\leq \left(\e^{-K_a(t,\pmb{\xi}_N)}+\e^{-K_a(t,\pmb{\xi}_M)}\right)|K_a(t,\pmb{\xi}_N)-K_a(t,\pmb{\xi}_M)|. \]
By Cauchy-Schwarz inequality and hypothesis H4,
\begin{align*}
 \mathbb{E}[|Y_N-Y_M|]\leq {} & \left(\|\e^{-K_a(t,\pmb{\xi}_N)}\|_{\leb^2(\Omega)}+\|\e^{-K_a(t,\pmb{\xi}_M)}\|_{\leb^2(\Omega)}\right)\|K_a(t,\pmb{\xi}_N)-K_a(t,\pmb{\xi}_M)\|_{\leb^2(\Omega)} \\
\leq {} & 2C \|K_a(t,\pmb{\xi}_N)-K_a(t,\pmb{\xi}_M)\|_{\leb^2(\Omega)}.
\end{align*}
By (\ref{ka}),
\[ \mathbb{E}[|Y_N-Y_M|]\leq 2C\sqrt{T-t_0}\|a_N-a_M\|_{\leb^2([t_0,T]\times\Omega)}. \]
As $a_N\rightarrow a$ in $\leb^2([t_0,T]\times\Omega)$, we conclude that $\mathbb{E}[|Y_N-Y|]\rightarrow0$ as $N\rightarrow\infty$. 

Thus, $\lim_{N\rightarrow\infty}f_1^N(x,t)=g(x,t)$, as wanted. We need to ensure that $g(x,t)$ is a density of $x(t,\omega)$. We know, by the proof of Theorem \ref{teor1}, that truncation (\ref{xNM}), $x_{N,N}(t,\omega)$, tends a.s., hence in law, to $x(t,\omega)$, for all $t\in [t_0,T]$. However, we cannot conclude as in the end of the proof of Theorem \ref{teor1}, because we do not have uniform convergence in order to justify the step from (\ref{FTCD}) to (\ref{FTCD2}). We need an alternative. First, notice that $g(\cdot,t)$ is a density function for each $t\in[t_0,T]$, because
\small
\[ \int_\mathbb{R} g(x,t)\,\dif x=\int_\mathbb{R} \mathbb{E}[f_0(xY-Z)Y]\,\dif x=\mathbb{E}\left[\int_\mathbb{R} f_0(xY-Z)Y\,\dif x\right]=\mathbb{E}\left[\int_\mathbb{R} f_0(x)\,\dif x\right]=1. \]
\normalsize
Now, let $y(t,\omega)$ be a random variable with law given by the density $g(x,t)$. By Scheffé's Lemma \cite[p.55]{scheffe}, $x_{N,N}(t,\omega)$ tends in law to $y(t,\omega)$. Therefore, $y(t,\omega)$ and $x(t,\omega)$ are equal in distribution (the limit in law is unique), so $x(t,\omega)$ is absolutely continuous with density function $f_1(x,t)=g(x,t)$, as wanted.
\end{proof}

As a consequence, for the homogeneous problem (\ref{edo_det}) with $b=0$ we have (recall the definition of $D(x_0)$ in (\ref{Dx0})):
\begin{theorem} \label{teornou2}
Assume the following four hypotheses:
\begin{align*}
 \text{H1}: {} & \;a\in \leb^2([t_0,T]\times\Omega); \\
 \text{H2}: {} & \;x_0 \text{ and }(\xi_1,\ldots,\xi_N) \text{ are absolutely continuous and independent, }N\geq1; \\
 \text{H3}: {} & \;\text{the density function of }x_0\text{, }f_0 \text{, is continuous and bounded on }D(x_0); \\ 
\text{H4}: {} & \;\|\e^{-K_a(t,\pmb{\xi}_N)}\|_{\leb^2(\Omega)}\leq C,\text{ for all }N\geq1\text{ and }t\in [t_0,T].
\end{align*}
Then, for all $x\in\mathbb{R}$ and $t\in [t_0,T]$, the sequence $\{f_1^N(x,t)\}_{N=1}^\infty$ given in (\ref{f1Nhomo}) converges to a density $f_1(x,t)$ of the solution (\ref{xhomo}).
\end{theorem}

For the random homogeneous problem (\ref{edo_det}) with $b=0$, the following theorem imposes a growth condition on $f_0$ (which is usually accomplished), so that hypothesis H4 of Theorem \ref{teornou2} can be avoided. Probably, this is the most general theorem on the random non-autonomous homogeneous linear differential equation in this paper:
\begin{theorem} \label{avo}
Assume the following three hypotheses:
\begin{align*}
 \text{H1}: {} & \;a\in \leb^2([t_0,T]\times\Omega); \\
 \text{H2}: {} & \;x_0 \text{ and }(\xi_1,\ldots,\xi_N) \text{ are absolutely continuous and independent, }N\geq1; \\
 \text{H3}: {} & \;\text{the density function of }x_0\text{, }f_0 \text{, is continuous on }D(x_0) \\ 
{} & \,\text{and }f_0(x)\leq C/|x|,\text{ for all }0\neq x\in D(x_0).
\end{align*}
Then, for all $0\neq x\in\mathbb{R}$ and $t\in [t_0,T]$, the sequence $\{f_1^N(x,t)\}_{N=1}^\infty$ given in (\ref{f1Nhomo}) converges to a density $f_1(x,t)$ of the solution (\ref{xhomo}).
\end{theorem}
\begin{proof}
Fix $0\neq x\in\mathbb{R}$ and $t\in [t_0,T]$. By (\ref{f1Nhomo}), we may write $f_1^N(x,t)=\mathbb{E}[f_0(x Y_N)Y_N]$, where $Y_N(\omega)$ and $Y(\omega)$ are the same as in the proof of Theorem \ref{teornou1}. Recall that $Y_N(\omega)\rightarrow Y(\omega)$ as $N\rightarrow\infty$, for a.e. $\omega\in\Omega$. By the continuity of $f_0$ on $D(x_0)$, $f_0(x Y_N(\omega))Y_N(\omega)\rightarrow f_0(x Y(\omega))Y(\omega)$ as $N\rightarrow\infty$, for a.e. $\omega\in\Omega$. Since
\[ f_0(x Y_N)Y_N=\frac{1}{|x|} f_0(x Y_N)|x|Y_N\leq \frac{C}{|x|}, \]
the Dominated Convergence Theorem does the rest: 
\[ \lim_{N\rightarrow\infty}f_1^N(x,t)=\mathbb{E}[f_0(xY)Y]=:g(x,t). \]
As in the end of the proof of Theorem \ref{teornou1}, one shows that $g(x,t)$ is a density function of the solution stochastic process $x(t,\omega)$.
\end{proof}

A second result for the random complete linear differential equation is the following reformulation of Theorem \ref{teor3}. The Lipschitz hypothesis is substituted by continuity and boundedness, although the uniform convergence is replaced by pointwise convergence.

\begin{theorem} \label{teornou3}
Assume that
\begin{align*}
 \text{H1}: {} & \;a,b\in \leb^2([t_0,T]\times\Omega); \\
 \text{H2}: {} & \;x_0,\,\eta_1,\,(\xi_1,\ldots,\xi_N,\eta_2,\ldots,\eta_N) \text{ are absolutely continuous and independent, }N\geq1; \\
 \text{H3}: {} & \;\text{the density function of }\eta_1\text{, }f_{\eta_1} \text{, is continuous and bounded on }\mathbb{R}; \\
 \text{H4}: {} & \;\xi_1,\xi_2,\ldots \text{ have compact support in }[-A,A]\, (A>0) \text{ and }\psi_1>0\text{ on }(t_0,T). 
\end{align*}
Then, for each fixed $t\in (t_0,T]$ and $x\in\mathbb{R}$, the sequence $\{f_1^N(x,t)\}_{N=1}^\infty$ given in (\ref{f1nou}) (which is the same as (\ref{f1n})) converges to a density $f_1(x,t)$ of the solution (\ref{sol_ale}).
\end{theorem}
\begin{proof}
Fix $x\in\mathbb{R}$ and $t\in (t_0,T]$. From the expression (\ref{f1nou}), define
\[ Y_N(\omega):=\frac{\e^{-K_a(t,\pmb{\xi}_N(\omega))}}{\sqrt{\gamma_1}\int_{t_0}^t \psi_1(s)\e^{-K_a(s,\pmb{\xi}_N(\omega))}\,\dif s} \]
and 
\[ Z_N(\omega):=\frac{x\e^{-K_a(t,\pmb{\xi}_N(\omega))}-x_0(\omega)-\int_{t_0}^t\left(\mu_b(s)+\sum_{i=2}^N\sqrt{\gamma_i}\,\psi_i(s)\eta_i(\omega)\right)\e^{-K_a(s,\pmb{\xi}_N(\omega))}\,\dif s}{\sqrt{\gamma_1}\int_{t_0}^t \psi_1(s)\e^{-K_a(s,\pmb{\xi}_N(\omega))}\,\dif s}. \]
As in (\ref{lima}) and (\ref{limb}) (details are left),
\[ \lim_{N\rightarrow\infty} Y_N(\omega)= \frac{\e^{-\int_{t_0}^t a(s,\omega)\,\dif s}}{\sqrt{\gamma_1}\int_{t_0}^t \psi_1(s)\e^{-\int_{t_0}^s a(r,\omega)\,\dif r}\,\dif s}=:Y(\omega) \]
and
\small
\begin{align*}
 \lim_{N\rightarrow\infty} Z_N(\omega)= {} & \frac{x\e^{-\int_{t_0}^t a(s,\omega)\,\dif s}-x_0(\omega)-\int_{t_0}^t\left(\mu_b(s)+b(s,\omega)-b_1(s,\omega)\right)\e^{-\int_{t_0}^s a(r,\omega)\,\dif r}\,\dif s}{\sqrt{\gamma_1}\int_{t_0}^t \psi_1(s)\e^{-\int_{t_0}^s a(r,\omega)\,\dif r}\,\dif s} \\
=:{}&Z(\omega), 
\end{align*}
\normalsize
for a.e. $\omega\in\Omega$. By (\ref{f1nou}), $f_1^N(x,t)=\mathbb{E}[f_{\eta_1}(Z_N)Y_N]$. As $\xi_1,\xi_2,\ldots$ live in $[-A,A]$, we can use bound (\ref{eka}) again to conclude that $|Y_N(\omega)|\leq B$ for certain $B>0$, for all $N$. By the continuity of $f_{\eta_1}$, $f_{\eta_1}(Z_N(\omega))Y_N(\omega)\rightarrow f_{\eta_1}(Z(\omega))Y(\omega)$ as $N\rightarrow\infty$, for a.e. $\omega\in\Omega$. As $|f_{\eta_1}(Z_N(\omega))Y_N(\omega)|\leq B\|f_{\eta_1}\|_{\leb^\infty(\mathbb{R})}$, by the Dominated Convergence Theorem, $\lim_{N\rightarrow\infty} f_1^N(x,t)=\mathbb{E}[f_{\eta_1}(Z)Y]=:g(x,t)$. To check that $g(x,t)$ is indeed a density of $x(t,\omega)$, one concludes as in the proof of Theorem \ref{teornou1}.
\end{proof}

A third and final result for the homogeneous problem is the following. It substitutes the Lipschitz hypothesis of Theorem \ref{teorb0L} by the weaker assumptions of continuity and boundedness, but we loose the uniform convergence on compact sets not containing $0$. The proof is very similar to that of Theorem \ref{teornou1}.
\begin{theorem} \label{teornoub0L}
Assume that:
\begin{align*}
 \text{H1}: {} & \;a\in \leb^2([t_0,T]\times\Omega); \\
 \text{H2}: {} & \;x_0,\;\xi_1\text{ and }(\xi_2,\ldots,\xi_N) \text{ are absolutely continuous and independent, }N\geq2; \\
 \text{H3}: {} & \;\text{the density function of }\xi_1\text{, }f_{\xi_1} \text{, is continuous and bounded on }\mathbb{R}; \\
 \text{H4}: {} & \; \int_{t_0}^t\phi_1(s)\,\dif s\neq0\text{ for all }t\in (t_0,T]; \\
 \text{H5}: {} & \; \text{either $x_0(\omega)>0$ for a.e. $\omega\in\Omega$ or $x_0(\omega)<0$ for a.e. $\omega\in\Omega$}.
\end{align*}
Then, for each fixed $t\in (t_0,T]$ and $x\neq0$, the sequence $\{f_1^N(x,t)\}_{N=1}^\infty$ given in (\ref{f1ncall}) (which is the same as (\ref{f1Nhomo})) converges to a density $f_1(x,t)$ of the solution stochastic process (\ref{xhomo}).
\end{theorem}
\begin{proof}
We assume $x_0(\omega)>0$ for a.e. $\omega\in\Omega$ (the other case is analogous). From expression (\ref{f1ncall}), we deduce that $f_1^N(x,t)=0$ if $x<0$. Thus, it suffices to consider the case $x>0$, so that (\ref{f1ncall}) becomes
\footnotesize
\begin{align}
 {} & f_1^N(x,t) \nonumber \\
= {} & \int_{\mathbb{R}^{N-1}\times (0,\infty)} f_{\xi_1}\left(\frac{1}{\sqrt{\nu_1}\,\int_{t_0}^t\phi_1(s)\,\dif s} \left\{\log\left(\frac{x}{x_0}\right)-\int_{t_0}^t\mu_a(s)\,\dif s-\sum_{j=2}^N\sqrt{\nu_j}\,\left(\int_{t_0}^t\phi_j(s)\,\dif s\right)\xi_j\right\}\right) \nonumber \\
\cdot {} & f_{(\xi_2,\ldots,\xi_N)}(\xi_2,\ldots,\xi_N)f_0(x_0)\frac{1}{|x|\sqrt{\nu_1}\,\left|\int_{t_0}^t \phi_1(s)\,\dif s\right|}\,\dif \xi_2\cdots\dif \xi_N\,\dif x_0, \label{f1ncall2}
\end{align}
\normalsize
Let
\small
\[ Y_N(\omega):=\frac{1}{\sqrt{\nu_1}\,\int_{t_0}^t\phi_1(s)\,\dif s} \left\{\log\left(\frac{x}{x_0(\omega)}\right)-\int_{t_0}^t\mu_a(s)\,\dif s-\sum_{j=2}^N\sqrt{\nu_j}\,\left(\int_{t_0}^t\phi_j(s)\,\dif s\right)\xi_j(\omega)\right\}. \]
\normalsize
Then, by (\ref{f1ncall2}),
\[ f_1^N(x,t)=\frac{1}{|x|\sqrt{\nu_1}\,\left|\int_{t_0}^t \phi_1(s)\,\dif s\right|}\mathbb{E}[f_{\xi_1}(Y_N)]. \]
We have that
\small
\begin{align*}
 \lim_{N\rightarrow \infty} Y_N(\omega)= {} & \frac{1}{\sqrt{\nu_1}\,\int_{t_0}^t\phi_1(s)\,\dif s} \left\{\log\left(\frac{x}{x_0(\omega)}\right)-\int_{t_0}^t\mu_a(s)\,\dif s-\int_{t_0}^t (a(s,\omega)-a_1(s,\omega))\,\dif s\right\} \\
=:{}&Y(\omega), 
\end{align*}
\normalsize
for a.e. $\omega\in\Omega$. 

One could use again the Dominated Convergence Theorem, as in Theorem \ref{teornou1}, to conclude. We present another reasoning which we believe is interesting. We use the following facts: the a.s. limit implies limit in law, and a sequence $\{X_n\}_{n=1}^\infty$ of random vectors with $p$ components tends in law to $X$ if and only if $\lim_{n\rightarrow\infty}\mathbb{E}[H(X_n)] =\mathbb{E}[H(X)]$ for every continuous and bounded map $H$ on $\mathbb{R}^p$ \cite[p.144 (iii)]{counter}. Taking 
\[ H(y):=\frac{1}{|x|\sqrt{\nu_1}\,\left|\int_{t_0}^t \phi_1(s)\,\dif s\right|}f_{\xi_1}(y), \]
which is continuous and bounded on $\mathbb{R}$, we deduce that 
\[ \lim_{N\rightarrow\infty} f_1^N(x,t)=\mathbb{E}[H(Y)]=:g(x,t). \]
This reasoning uses strongly the convergence in law, rather than the a.s. convergence necessary to apply the Dominated Convergence Theorem. These ideas are useful for random differential equations in which the truncation converges in law, but not a.s. (for the random linear differential equation (\ref{edo_det}) this is not the case, so the Dominated Convergence Theorem is applicable, as we did).
\end{proof}

\subsection{Comments on the hypotheses of the theorems} \label{comments} \ \\

The hypotheses of the theorems in Subsection \ref{su5} are weaker than the hypotheses of the theorems from the previous subsections. In the theorems from Subsection \ref{su5}, one has to put conditions so that the Dominated Convergence Theorem can be applied in an expectation, essentially. This yields a pointwise convergence of the approximating sequence of density functions. However, when dealing with uniform convergence in Subsection \ref{su1}, \ref{su2}, \ref{su3} and \ref{su4}, the Lipschitz condition plus other assumptions are necessary. 

In terms of numerical experiments, every time Theorems \ref{teor1}, \ref{teor2}, \ref{teor3} and \ref{teorb0L} are applicable, Theorems \ref{teornou1}, \ref{teornou2}, \ref{teornou3} and \ref{teornoub0L} are applicable too, respectively.

The Lipschitz (or continuity) condition on $\mathbb{R}$ is satisfied by the probability density function of some named distributions:
\begin{itemize}
\item $\text{Normal}(\mu,\sigma^2)$, $\mu\in\mathbb{R}$ and $\sigma^2>0$.
\item $\text{Beta}(\alpha,\beta)$, $\alpha,\beta\geq2$.
\item $\text{Gamma}(\alpha,\beta)$, $\alpha\geq2$ and $\beta>0$.
\end{itemize}
In general, any density with bounded derivative on $\mathbb{R}$ satisfies the Lipschitz condition on $\mathbb{R}$, by the Mean Value Theorem. 

Some non-Lipschitz (and non-continuous) density functions on $\mathbb{R}$ are the uniform distribution, the exponential distribution... or any other density with a jump discontinuity at some point of $\mathbb{R}$. Notice however that, if $x_0$ is an exponentially distributed initial condition, then $f_0$ is Lipschitz on $D(x_0)=(0,\infty)$. Nevertheless, if $x_0$ is a uniform random variable, then $f_0$ remains being non-continuous on $D(x_0)$.

In Theorem \ref{teor3} and Theorem \ref{teorb0L}, for instance, we included the hypotheses $f_{\eta_1}$ and $f_{\xi_1}$ Lipschitz, respectively. It must be clear that $\xi_1$ and $\eta_1$ are not important, in the sense that, if some $\xi_k$ or $\eta_l$ satisfies the hypotheses for certain $k$ and $l$ not equal to $1$, then we may reorder the pairs of eigenvalues and eigenfunctions of the Karhunen-Loève expansions of $a$ and $b$ so that the necessary hypotheses hold. 

Let us see examples of processes $a$ for which, given any $c\in\mathbb{R}$, there is a constant $C>0$ such that the inequality $\mathbb{E}[\e^{c\,K_a(t,\pmb{\xi}_N)}]\leq C$ holds for all $N\geq1$ and $t\in [t_0,T]$:
\begin{itemize}
\item Case $\xi_1,\xi_2,\ldots$ are independent and $\text{Normal}(0,1)$ distributed, that is, when $a$ is a Gaussian process. Indeed, taking into account that the moment generating function of a Gaussian random variable $X$ with zero expectation and unit variance is given by $\mathbb{E}[\e^{\lambda X}]=\e^{\lambda^2/2}$, and using Cauchy-Schwarz inequality, one deduces
\begin{align*}
\mathbb{E}[\e^{c\,K_a(t,\pmb{\xi}_N)}]= {} & \e^{c\int_{t_0}^t \mu_a(s)\,\dif s}\mathbb{E}\left[\prod_{j=1}^N \e^{c\,\xi_j\sqrt{\nu_j}\, \int_{t_0}^t\phi_j(s)\,\dif s}\right]\\
= {} & \e^{c\int_{t_0}^t \mu_a(s)\,\dif s}\prod_{j=1}^N \mathbb{E}\left[\e^{c\,\xi_j\sqrt{\nu_j}\, \int_{t_0}^t\phi_j(s)\,\dif s}\right]\\
= {} & \e^{c\int_{t_0}^t \mu_a(s)\,\dif s}\prod_{j=1}^N \e^{\frac12 c^2\nu_j\left(\int_{t_0}^t\phi_j(s)\,\dif s\right)^2}\\
\leq {} & \e^{c\int_{t_0}^t \mu_a(s)\,\dif s} \e^{\frac12 c^2(T-t_0)\sum_{j=1}^N\nu_j\int_{t_0}^t\phi_j(s)^2\,\dif s}\\
\leq {} & \e^{c\int_{t_0}^t \mu_a(s)\,\dif s} \e^{\frac12 c^2(T-t_0)\int_{t_0}^T\left(\sum_{j=1}^N\nu_j\phi_j(s)^2\right)\,\dif s}.
\end{align*}
Now we bound each of the terms. First observe that, by Cauchy-Schwarz inequality,
\[\mu_a(s)^2=\mathbb{E}[a(s)]^2\leq \mathbb{E}[a(s)^2], \]
therefore, by Cauchy-Schwarz inequality again, one gets
\small
\begin{equation}
\int_{t_0}^T|\mu_a(s)|\,\dif s\leq \sqrt{T-t_0}\,\left(\int_{t_0}^T\mu_a(s)^2\,\dif s\right)^{\frac12}\leq \sqrt{T-t_0}\,\|a\|_{\leb^2([t_0,T]\times \Omega)}<\infty. 
 \label{intmua}
\end{equation}
\normalsize
For the other term, we have
\small
\begin{align*} 
\int_{t_0}^T\left(\sum_{j=1}^N\nu_j\phi_j(s)^2\right)\,\dif s\leq {} & \int_{t_0}^T\left(\sum_{j=1}^\infty\nu_j\phi_j(s)^2\right)\,\dif s \\
= {} & \int_{t_0}^T\left(\mathbb{E}[a(s)^2]-\mu_a(s)^2\right)\,\dif s \\
= {} & \|a\|_{\leb^2([t_0,T]\times \Omega)}^2-\|\mu_a\|_{\leb^2([t_0,T])}^2<\infty. 
\end{align*}
\normalsize

\item Case $\xi_1,\xi_2,\ldots$ are compactly supported in $[\alpha,\beta]$. In this case,
\begin{align*}
\quad \e^{c\,K_a(t,\pmb{\xi}_N(\omega))}= {} & \e^{c\int_{t_0}^t\left(\mu_a(s)+\sum_{j=1}^N \sqrt{\nu_j}\phi_j(s)\xi_j(\omega)\right)\,\dif s}\leq \e^{c\int_{t_0}^T |\mu_a(s)|\,\dif s} \e^{c\sum_{j=1}^N\sqrt{\nu_j}\left|\int_{t_0}^t \phi_j(s)\,\dif s\right||\xi_j(\omega)|} \\
\leq {} & \e^{c\int_{t_0}^T |\mu_a(s)|\,\dif s}\e^{c\max\{|\alpha|,|\beta|\}\sum_{j=1}^N \sqrt{\nu_j}\left|\int_{t_0}^t \phi_j(s)\,\dif s\right|} \\
\leq {} & \e^{c\int_{t_0}^T |\mu_a(s)|\,\dif s}\e^{c\max\{|\alpha|,|\beta|\}\sum_{j=1}^\infty \sqrt{\nu_j}\left|\int_{t_0}^t \phi_j(s)\,\dif s\right|}. 
\end{align*}
By (\ref{intmua}), $\int_{t_0}^T |\mu_a(s)|\,\dif s<\infty$. By Remark \ref{rmk}, 
\[ \sup_{t\in[t_0,T]}\sum_{j=1}^\infty \sqrt{\nu_j}\left|\int_{t_0}^t \phi_j(s)\,\dif s\right|<\infty. \] 
Thereby, there is a constant $C>0$ such that $\e^{c\,K_a(t,\pmb{\xi}_N(\omega))}\leq C$, for a.e. $\omega\in\Omega$, $N\geq1$ and $t\in [t_0,T]$. Applying expectations, $\mathbb{E}[\e^{c\,K_a(t,\pmb{\xi}_N(\omega))}]\leq C$, for every $N\geq1$ and $t\in [t_0,T]$.

If we add the hypothesis of independence to compactly supported in $[\alpha,\beta]$, the proof is analogous to the normal case. But instead of using the generating moment function of a normal distribution, one has to use Hoeffding's Lemma: ``Let $X$ be a random variable with $\mathbb{E}[X]=0$ and support in $[\alpha,\beta]$. Then, for all $\lambda \in \mathbb{R}$, $\mathbb{E}[\e^{\lambda X}]\leq \e^{\frac{\lambda^2(\beta-\alpha)^2}{8}}$''. For a proof, see \cite[p.21]{hoeffding}.
\end{itemize}

Let us see examples of processes $b$ for which hypothesis H4 from Theorem \ref{teor1} holds:
\[\|\mu_b\|_{\leb^p(t_0,T)}+\sum_{j=1}^\infty\sqrt{\gamma_j}\,\|\psi_j\|_{\leb^p(t_0,T)}\|\eta_j\|_{\leb^p(\Omega)}<\infty\]
for some $2< p\leq\infty$.
\begin{itemize}
\item Consider $\{b(t,\omega):\, t_0=0\leq t\leq T=1\}$ a standard Brownian motion on $[0,1]$. By Exercise 5.12 in \cite[p.~206]{llibre_powell}, we know that 
\begin{equation}
\gamma_{j}=\frac{1}{\left(j-\frac12\right)^2\pi^2}, \quad \psi_{j}(t)=\sqrt{2}\,\sin\left(t\left(j-\frac12\right)\pi\right), \quad j\geq1, 
 \label{BM}
\end{equation}
and $\eta_1,\eta_2,\ldots$ are $\text{Normal}(0,1)$ and independent random variables. We check that we can choose $p=3$ (so $q=12$). We need to check that $\sum_{j=1}^{\infty} \sqrt{\gamma_j}\,\|\psi_j\|_{\leb^3(0,1)}<\infty$.
We have that
\small
\[\;\;\;\;\;\;\;\|\psi_j\|_{\leb^3(0,1)}=\sqrt{2}\,\left(\int_0^1 \left|\sin\left(t\left(j-\frac12\right)\pi\right)\right|^3\,\dif t\right)^{\frac13}=2\,\left(\frac{\sqrt{2}}{3\pi(2j-1)}\right)^{\frac13}\sim \frac{1}{j^{\frac13}},  \]
\normalsize
therefore $\sqrt{\gamma_j}\,\|\psi_j\|_{\leb^3(0,1)}\sim 1/j^{1+\frac{1}{3}}$, and since $\sum_{j=1}^\infty 1/j^{1+\frac{1}{3}}<\infty$, by comparison we obtain $\sum_{j=1}^{\infty} \sqrt{\gamma_j}\,\|\psi_j\|_{\leb^3(0,1)}<\infty$, as wanted.

\item Consider $\{b(t,\omega):\, t_0=0\leq t\leq T=1\}$ a standard Brownian bridge on $[0,1]$. By Example 5.30 in \cite[p.~204]{llibre_powell}, we know that 
\begin{equation}
 \gamma_j=\frac{1}{j^2\pi^2}, \quad \psi_j(t)=\sqrt{2}\,\sin(t j\pi), \quad j\geq1, 
 \label{BB}
\end{equation}
and $\eta_1,\eta_2,\ldots$ are $\text{Normal}(0,1)$ and independent random variables. As before, one can take $p=3$.
\end{itemize}

Finally, we want to make some comments on the hypotheses $\psi_1>0$ on $(t_0,T)$ from Theorem \ref{teor3} and Theorem \ref{teornou3}, and $\int_{t_0}^t \phi_1(s)\,\dif s\neq 0$ for all $t\in (t_0,T]$ from Theorem \ref{teorb0L} and Theorem \ref{teornoub0L}. There are some eigenvalue problems in the theory of deterministic differential equations in which one can ensure the existence of a positive eigenfunction for the greatest eigenvalue (for instance, Sturm-Liouville problems, see \cite{SL1}; the Poisson equation with Dirichlet boundary conditions, see \cite[p.452]{salsa}; etc.). We ask ourselves if the same holds with the operator (\ref{karhC}) of the Karhunen-Loève expansion.

Consider a generic Karhunen-Loève expansion (\ref{karh}), $\mu(t)+\sum_{j=1}^\infty \sqrt{\nu_j}\,\phi_j(t)\xi_j(\omega)$, of a stochastic process $X\in\leb^2(\mathcal{T}\times\Omega)$, $\mathcal{T}\subseteq\mathbb{R}$. Suppose in this particular discussion that $\nu_1$ is the greatest eigenvalue (in general it may not, because one can order the pairs of eigenvalues and eigenfunctions in the Karhunen-Loève expansion as desired). In general, one cannot ensure whether $\nu_1$ has a positive eigenfunction $\phi_1$. For example, if $\mathcal{T}=[-\pi,\pi]$, we know that $\{1\}\cup \{\cos(tj),\sin(tj):\,j\geq1\}$ is an orthonormal basis of $\leb^2([-\pi,\pi])$, so we can consider the stochastic process $X(t,\omega)=\cos(t)\xi(\omega)$, where $\xi\sim \text{Normal}(0,1)$. The eigenvalues of the associated integral operator $\mathcal{C}$ given by (\ref{karhC}) are $0$ and $1$. The eigenspace associated to $0$ is infinite dimensional, whereas the eigenspace associated to $1$ is spanned by $\phi_1(t)=\cos(t)$. This eigenfunction takes positive and negative values on $(-\pi,\pi)$.

To gain some intuition on why this phenomenon occurs, recall that in a finite dimensional setting, the greatest eigenvalue of a matrix has an eigenvector with nonnegative entries if all entries of the matrix are nonnegative, by Perron-Frobenius Theorem, see  \cite[Th.~8.2.8 in p.~526; Th.~8.3.1, p.~529]{perron}. In our setting, there are some results in the literature that give conditions under which the eigenfunction $\phi_1$ can be taken nonnegative or positive. Let $K(t,s)=\Cov[X(t),X(s)]$ be the so-called kernel of the operator $\mathcal{C}$ given by (\ref{karhC}). In \cite[p.~6]{karlin}, Theorem 1 says that, if $K(t,s)\geq0$ for all $t,s\in\mathcal{T}$, then one can choose a nonnegative eigenfunction $\phi_1$ for the greatest eigenvalue $\nu_1$. In \cite[p.~7]{karlin}, Theorem 2 gives as a consequence that, if $\mathcal{T}$ is open and $K(t,s)>0$ for all $t,s\in\mathcal{T}$, then $\nu_1$ is simple and $\phi_1$ can be chosen positive on $\mathcal{T}$. This is what happens for example with the standard Brownian motion and the standard Brownian bridge on $[0,1]$. They have covariances $K(t,s)=\min\{t,s\}$ and $K(t,s)=\min\{t,s\}-ts$, respectively, which are positive on $(t,s)\in(0,1)\times(0,1)$, therefore the eigenfunction associated to the greatest eigenvalue can be picked positive on $(0,1)$. Indeed, take $\nu_1=4/\pi^2$ and $\phi_1(t)=\sqrt{2}\sin(\pi t/2)$ for the Brownian motion (see (\ref{BM})), and $\nu_1=1/\pi^2$ and $\phi_1(t)=\sqrt{2}\sin(\pi t)$ for the Brownian bridge (see (\ref{BB})). These eigenfunctions $\phi_1$ are positive on $(0,1)$.

\section{Numerical examples}

Under the hypotheses of Theorem \ref{teor1}, Theorem \ref{teor3}, Theorem \ref{teornou1} or Theorem \ref{teornou3}, the density function (\ref{f1n}),
\footnotesize
\[ 
 f_1^N(x,t)=\int_{\mathbb{R}^{2N}} f_0\left(x\,\e^{-K_a(t,\pmb{\xi}_N)}-\int_{t_0}^t S_b(s,\pmb{\eta}_N)\e^{-K_a(s,\pmb{\xi}_N)}\,\dif s\right)f_{\pmb{\xi}_N,\pmb{\eta}_N}(\pmb{\xi}_N,\pmb{\eta}_N)\e^{-K_a(t,\pmb{\xi}_N)}\,\dif \pmb{\xi}_N\,\dif \pmb{\eta}_N,
\]
\normalsize
gives an approximation of the density function of the solution process to the non-autonomous complete linear differential equation, $x(t,\omega)$, given in (\ref{sol_ale}). Under the hypotheses of Theorem \ref{teor3} or Theorem \ref{teornou3} the approximating density (\ref{f1nou}) is obviously valid as well (it is equal to (\ref{f1n})), but we believe that (\ref{f1n}) has a simpler expression.

On the other hand, under the hypotheses of Theorem \ref{teor2}, Theorem \ref{teorb0L}, Theorem \ref{teornou2}, Theorem \ref{avo} or Theorem \ref{teornoub0L}, the density function (\ref{f1Nhomo}),
\[ f_1^N(x,t)=\int_{\mathbb{R}^{N}} f_0\left(x\,\e^{-K_a(t,\pmb{\xi}_N)}\right)f_{\pmb{\xi}_N}(\pmb{\xi}_N)\e^{-K_a(t,\pmb{\xi}_N)}\,\dif \pmb{\xi}_N, \]
approximates the density of the solution process to the non-autonomous homogeneous linear differential equation, $x(t,\omega)$, given in (\ref{xhomo}). Under the assumptions of Theorem \ref{teorb0L} and Theorem \ref{teornoub0L}, the approximation via the density (\ref{f1ncall}) is valid too (it is equal to (\ref{f1Nhomo})), but we think that (\ref{f1Nhomo}) has a simpler expression.

That is why we will make use of (\ref{f1n}) (under the hypotheses of Theorem \ref{teor1}, Theorem \ref{teor3}, Theorem \ref{teornou1} or Theorem \ref{teornou3}), and of (\ref{f1Nhomo}) (under the hypotheses of Theorem \ref{teor2}, Theorem \ref{teorb0L}, Theorem \ref{teornou2}, Theorem \ref{avo} or Theorem \ref{teornoub0L}) to perform numerical approximations of the densities of the processes (\ref{sol_ale}) and (\ref{xhomo}), respectively. 

We will use the software Mathematica\textsuperscript{\tiny\textregistered} to perform the computations. The density function (\ref{f1Nhomo}) can be numerically computed in an exact manner with the built-in function \verb|NIntegrate|. Using this function, we have drawn the shape of (\ref{f1Nhomo}) in Example \ref{Brownia}, Example \ref{nogaussia} and Example \ref{x0lip}, which deal with the case $b=0$.

The computational time spent for (\ref{f1n}) is much larger than for (\ref{f1Nhomo}), due to the additional terms that $b$ involves: there is a parametric numerical integration in the evaluation of $f_0$ in (\ref{f1n}). For $N=1$, the function \verb|NIntegrate| has been successfully used in Example \ref{complete} and Example \ref{complete2} (case $b\neq0$). However, for $N\geq2$, the function \verb|NIntegrate| has been tried, with the consequence that numerical integration becomes unfeasible. Thereby, we use the following trick: (\ref{f1n}) is an integral with respect to a density function, so it may be seen as an expectation, as we did in (\ref{f1nexp}). The Law of Large Numbers allows approximating the expectation via sampling. This is the theoretical basis of Monte Carlo simulations \cite[p.53]{xiu}. Hence, there are two possibilities to compute (\ref{f1n}) for $N\geq2$: the built-in function \verb|NExpectation| with the option \verb|Method -> "MonteCarlo"| (we used this function for $N=2$ and $N=3$ in Example \ref{complete}), or programing oneself a Monte Carlo procedure (we did so for $N=2$ in Example \ref{complete2}, with $40,000$ realizations of each one of the random variables $\xi_i$ and $\eta_i$ taking part). 

We would like to remark that there are different ways one may carry out the computations in Mathematica\textsuperscript{\tiny\textregistered} or any other software. The best built-in function or programming procedure depends on each case, and we have just presented different settings showing how one may act.

\begin{example} \label{Brownia} \normalfont
Consider the homogeneous linear differential equation ($b=0$) with $x_0\sim\text{Uniform}(1,2)$ and $a(t,\omega)=W(t,\omega)$, where $W$ is a standard Brownian motion on $[t_0,T]=[0,1]$. The Karhunen-Loève expression of $a(t,\omega)$ is
\[ a(t,\omega)=\sum_{j=1}^{\infty} \frac{\sqrt{2}}{\left(j-\frac12\right)\pi}\sin\left(t\left(j-\frac12\right)\pi\right)\xi_j(\omega), \]
where $\xi_1,\xi_2,\ldots$ are independent and $\text{Normal}(0,1)$ random variables (see Exercise 5.12 in \cite{llibre_powell}).

We want to approximate the probability density function of the solution process $x(t,\omega)$ given in (\ref{xhomo}). The assumptions of Theorem \ref{teorb0L} hold, therefore (\ref{f1Nhomo}) is a suitable approximation of the density of $x(t,\omega)$. In Figure \ref{brownia}, we can see $f_1^N(x,t)$ for $N=1$ (up left), $N=2$ (up right) and $N=3$ (down) at $t=0.5$. In Figure \ref{brownia3D}, a three dimensional plot gives $f_1^N(x,t)$ for $N=2$, $x\in\mathbb{R}$ and $0\leq t\leq 1$. 

In fact, in this case it is possible to compute the exact density function of $x(t,\omega)$. Since $a(t,\omega)=W(t,\omega)$ is a continuous process, the Lebesgue integral $\int_0^t a(s,\omega)\,\dif s$ turns out to be a Riemann integral, for each fixed $\omega\in\Omega$. Therefore it is an a.s. limit of Riemann sums. Each Riemann sum is a normal distributed random variable (because the process $a$ is Gaussian), and since the a.s. limit of normal random variables is normal again, we obtain that $Z_t:=\int_0^t a(s,\omega)\,\dif s$ is normally distributed. Its expectation is $\mathbb{E}[\int_0^t a(s,\omega)\,\dif s]=\int_0^t \mathbb{E}[a(s,\omega)]\,\dif s=0$ and its variance is $\mathbb{V}[\int_0^t a(s,\omega)\,\dif s]=\mathbb{E}[(\int_0^t a(s,\omega)\,\dif s)^2]=\mathbb{E}[\int_0^t\int_0^t a(s,\omega)a(r,\omega)\,\dif s\,\dif r]=\int_0^t\int_0^t \mathbb{E}[a(s,\omega)a(r,\omega)]\,\dif s\,\dif r=\int_0^t \int_0^t \min\{s,r\}\,\dif s\,\dif r=t^3/3$.
Using Lemma \ref{lema_abscont}, we can compute the density of $x(t,\omega)$. Indeed, in the notation of Lemma \ref{lema_abscont}, let $g(x_0,Z_t)=(x_0\e^{Z_t},Z_t)$, $D=\mathbb{R}^2$, $g(D)=\mathbb{R}^2$, $h(u,v)=(u\e^{-v},v)$ and $Jh(u,v)=\e^{-v}\neq0$. Then we compute the density of $(x_0\e^{Z_t},Z_t)$, and taking the marginal distribution,
\begin{equation}
 f_1(x,t)=\int_{\mathbb{R}} f_0(x\e^{-y})f_{\text{Normal}(0,t^3/3)}(y)\e^{-y}\,\dif y. 
 \label{f1e1}
\end{equation}
In Figure \ref{real} we present a plot of this density for $t=0.5$ (left) and a three dimensional plot for $x\in\mathbb{R}$ and $0\leq t\leq 1$ (right). In this way we can compare visually the approximation performed via $f_1^N(x,t)$ in Figure \ref{brownia} and Figure \ref{brownia3D}.

To assess analytically the accuracy of the approximations $f_1^N(x,t)$ with respect to the exact probability density function $f_1(x,t)$ given in (\ref{f1e1}), in Table \ref{table1} we collect the values of the errors for different orders of truncation $N$ at $t=0.5$. We observe that the error decreases as $N$ increases, so this agrees with our theoretical findings.

\begin{figure}[H]
  \begin{center}
    \includegraphics[width=7cm]{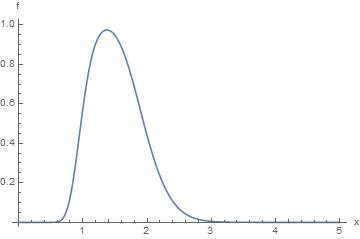}
		\includegraphics[width=7cm]{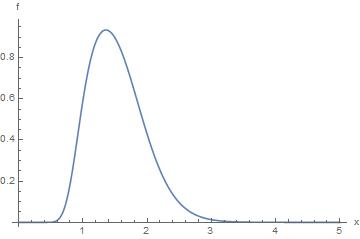}
		\includegraphics[width=7cm]{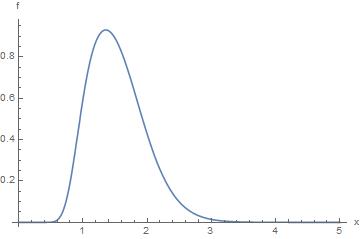}
    \caption{Density $f_1^N(x,t)$ for $N=1$ (up left), $N=2$ (up right) and $N=3$ (down) at the point $t=0.5$. Example \ref{Brownia}.}
		\label{brownia}
    \end{center}
  \end{figure}
	
\begin{figure}[H]
  \begin{center}
    \includegraphics[width=6cm]{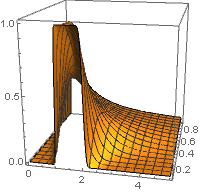}
    \caption{Density $f_1^N(x,t)$ for $N=2$, $x\in\mathbb{R}$ and $0\leq t\leq 1$. Example \ref{Brownia}.}
		\label{brownia3D}
    \end{center}
  \end{figure}

\begin{figure}[H]
  \begin{center}
	\includegraphics[width=8cm]{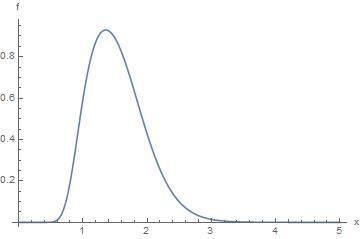}
    \includegraphics[width=6cm]{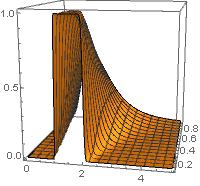}
    \caption{Density $f_1(x,t)$ for $t=0.5$ (left) and three dimensional plot for $x\in\mathbb{R}$ and $0\leq t\leq 1$ (right). Example \ref{Brownia}.}
		\label{real}
    \end{center}
  \end{figure}
	
\begin{table}[H]
\begin{center}
\begin{tabular}{|c|c|} \hline
$N$ & $\|f_1^N(x,0.5)-f_1(x,0.5)\|_{\leb^\infty(\mathbb{R})}$ \\ \hline
$1$ & $0.0687343$ \\ \hline
$2$ & $0.00743475$ \\ \hline
$3$ & $0.00332728$ \\ \hline
\end{tabular}
\caption{Error with respect to the exact density function $f_1(x,t)$ given by \eqref{f1e1} for $N=1$, $N=2$ and $N=3$, at $t=0.5$. Example \ref{Brownia}.}
\label{table1}
\end{center}
\end{table}
	
\end{example}

\begin{example} \label{nogaussia} \normalfont
Again, consider the homogeneous linear differential equation ($b=0$) with $x_0\sim\text{Uniform}(1,2)$, but now $a$ will not be a Gaussian process. We consider 
\[ a(t,\omega)=\sum_{j=1}^\infty \frac{\sqrt{2}}{j}\sin(j\pi t)\xi_j(\omega), \]
where $[t_0,T]=[0,1]$ and $\xi_1,\xi_2,\ldots$ are independent random variables with common density function $f_{\xi_1}(\xi_1)=\sqrt{2}/(\pi(1+\xi_1^4))$.
These random variables have zero expectation and unit variance, so we have a proper expansion for $a$ (its Karhunen-Loève expansion).

We want to approximate the probability density function of the solution process $x(t,\omega)$ given in (\ref{xhomo}). The assumptions of Theorem \ref{teorb0L} hold, because $f_{\xi_1}$ is a Lipschitz function on $\mathbb{R}$ (hypothesis H3) and $\phi_1(t)=\sqrt{2}\,\sin(t\pi)>0$ on $(0,1)$ (hypothesis H4). Hence, (\ref{f1Nhomo}) is a suitable approximation of the density of the process $x(t,\omega)$ given in (\ref{xhomo}). In Figure \ref{nog}, we can see $f_1^N(x,t)$ for $N=1$ (up left), $N=2$ (up right) and $N=3$ (down) at $t=0.7$. In Figure \ref{nog3D}, a three dimensional plot gives $f_1^N(x,t)$ for $N=2$, $x\in\mathbb{R}$ and $0\leq t\leq 1$. In Table \ref{table2} we have written the difference between two consecutive orders of truncation $N$ at $t=0.7$. Observe that these differences decrease to $0$, which agrees with the theoretical results.

\begin{figure}[H]
  \begin{center}
    \includegraphics[width=6.7cm]{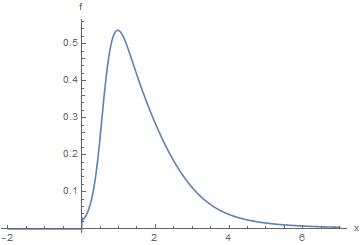}
		\includegraphics[width=6.7cm]{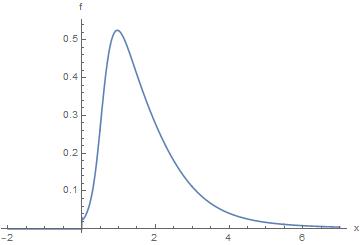}
		\includegraphics[width=6.7cm]{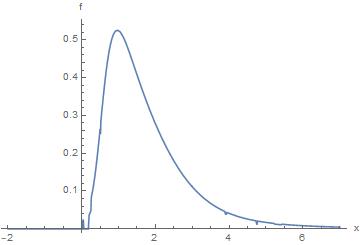}
    \caption{Density $f_1^N(x,t)$ for $N=1$ (up left), $N=2$ (up right) and $N=3$ (down) at the point $t=0.7$. Example \ref{nogaussia}.}
		\label{nog}
    \end{center}
  \end{figure}
	
\begin{table}[H]
\begin{center}
\begin{tabular}{|c|c|} \hline
$N$ & $\|f_1^N(x,0.7)-f_1^{N+1}(x,0.7)\|_{\leb^\infty(\mathbb{R})}$ \\ \hline
$1$ & $0.010764$ \\ \hline
$2$ & $0.000177$ \\ \hline
\end{tabular}
\caption{Error between two consecutive orders of truncation $N$ and $N+1$, for $N=1$ and $N=2$, at $t=0.7$. Example \ref{nogaussia}.}
\label{table2}
\end{center}
\end{table}
	
\begin{figure}[H]
  \begin{center}
    \includegraphics[width=6cm]{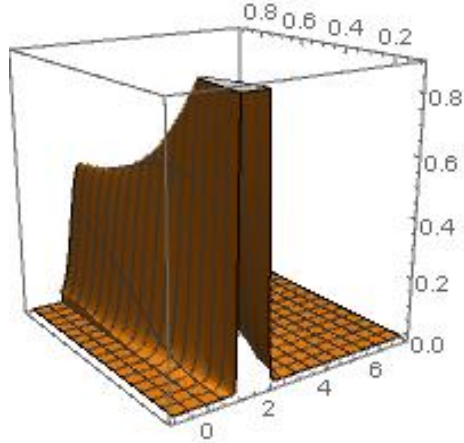}
    \caption{Density $f_1^N(x,t)$ for $N=2$, $x\in\mathbb{R}$ and $0\leq t\leq 1$. Example \ref{nogaussia}.}
		\label{nog3D}
    \end{center}
  \end{figure}
	 
\end{example}

\begin{example} \label{x0lip} \normalfont

Consider the homogeneous linear differential equation ($b=0$) with $x_0\sim\text{Beta}(5,6)$, $[t_0,T]=[0,1]$ and
\[ a(t,\omega)=-1+\sum_{j=1}^\infty \frac{\sqrt{2}}{j}\sin(j\pi t)\xi_j(\omega), \]
where $\xi_1,\xi_2,\ldots$ are independent random variables with uniform distribution on $(-\sqrt{3},\sqrt{3})$ (they have zero expectation and unit variance). Notice that $a$ is not a Gaussian process.

The hypotheses of Theorem \ref{teor2} hold, because the density function of a $\text{Beta}(5,6)$ random variable is Lipschitz on $\mathbb{R}$ and $\mathbb{E}[\e^{c\,K_a(t,\pmb{\xi}_N)}]\leq C$ for all $n\geq1$, $t\in [t_0,T]$ and $c\in \mathbb{R}$ (see Subsection \ref{comments}). Then (\ref{f1Nhomo}) approximates the density of the process $x(t,\omega)$ given in (\ref{xhomo}). 

In Figure \ref{lip}, we can see $f_1^N(x,t)$ for $N=1$ (up left), $N=2$ (up right), $N=3$ (down left) and $N=4$ (down right) at $t=0.3$. In Figure \ref{lip3D}, a three dimensional plot gives $f_1^N(x,t)$ for $N=2$, $x\in\mathbb{R}$ and $0\leq t\leq 1$. To assess analytically the convergence, in Table \ref{table3} we show the difference between two consecutive orders of truncation $N$ at $t=0.3$. We observe that these differences decrease to $0$, which goes in the direction of our theoretical results.

\begin{figure}[H]
  \begin{center}
    \includegraphics[width=7cm]{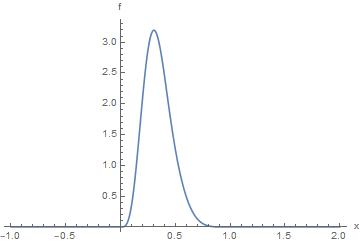}
		\includegraphics[width=7cm]{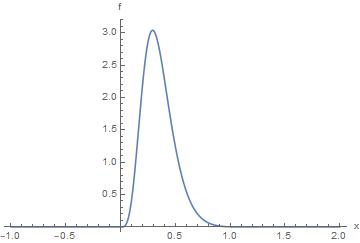}
		\includegraphics[width=7cm]{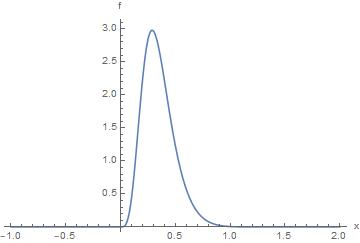}
		\includegraphics[width=7cm]{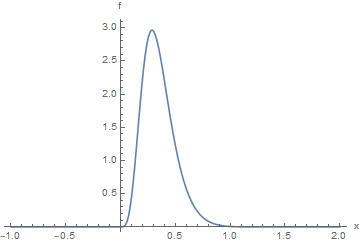}
    \caption{Density $f_1^N(x,t)$ for $N=1$ (up left), $N=2$ (up right), $N=3$ (down left) and $N=4$ (down right) at the point $t=0.3$. Example \ref{x0lip}.}
		\label{lip}
    \end{center}
  \end{figure}
	
\begin{table}[H]
\begin{center}
\begin{tabular}{|c|c|} \hline
$N$ & $\|f_1^N(x,0.3)-f_1^{N+1}(x,0.3)\|_{\leb^\infty(\mathbb{R})}$ \\ \hline
$1$ & $0.225333$ \\ \hline
$2$ & $0.0799602$ \\ \hline
$3$ & $0.0203143$ \\ \hline
\end{tabular}
\caption{Error between two consecutive orders of truncation $N$ and $N+1$, for $N=1$, $N=2$ and $N=3$, at $t=0.3$. Example \ref{x0lip}.}
\label{table3}
\end{center}
\end{table}
	
\begin{figure}[H]
  \begin{center}
    \includegraphics[width=6cm]{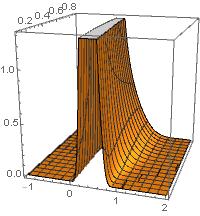}
    \caption{Density $f_1^N(x,t)$ for $N=2$, $x\in\mathbb{R}$ and $0\leq t\leq 1$. Example \ref{x0lip}.}
		\label{lip3D}
    \end{center}
  \end{figure}
	 
\end{example}

\begin{example} \label{complete} \normalfont
Consider a complete linear differential equation with $x_0\sim\text{Normal}(0,1)$, $a(t,\omega)=W(t,\omega)$ a standard Brownian motion on $[0,1]$ and $b(t,\omega)$ a standard Brownian bridge on $[0,1]$:
\[ a(t,\omega)=\sum_{j=1}^{\infty} \frac{\sqrt{2}}{\left(j-\frac12\right)\pi}\sin\left(t\left(j-\frac12\right)\pi\right)\xi_j(\omega), \quad  b(t,\omega)=\sum_{i=1}^{\infty} \frac{\sqrt{2}}{i\pi}\sin\left(ti\pi\right)\eta_i(\omega), \]
where $\xi_1,\xi_2,\ldots,\eta_1,\eta_2,\ldots$ are independent and $\text{Normal}(0,1)$ distributed random variables (see Example 5.30 and Exercise 5.12 in \cite[p.~204 and p.~216, resp.]{llibre_powell}).

The hypotheses of Theorem \ref{teor1} hold (see Subsection \ref{comments}). Therefore, the density $f_1^N(x,t)$ given in (\ref{f1n}) is an approximation of the density function of the process $x(t,\omega)$ given in (\ref{sol_ale}). 

In Figure \ref{comp} we have plotted $f_1^N(x,t)$ for $N=1$, $N=2$ and $N=3$ at $t=0.5$. 


\begin{figure}[H]
  \begin{center}
    \includegraphics[width=7cm]{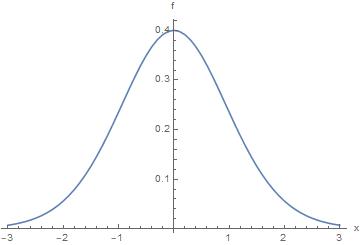}
		\includegraphics[width=7cm]{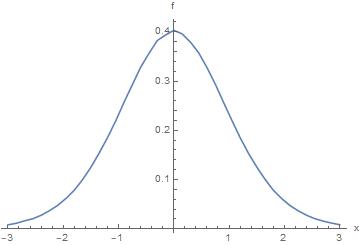}
		\includegraphics[width=7cm]{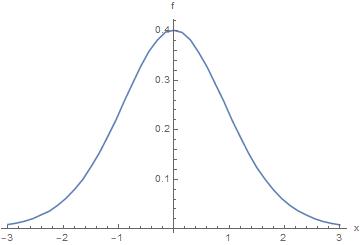}
    \caption{Density $f_1^N(x,t)$ for $N=1$ (up left), $N=2$ (up right) and $N=3$ (down) at the point $t=0.5$. Example \ref{complete}.}
		\label{comp}
    \end{center}
  \end{figure}

\end{example}

\begin{example} \label{complete2} \normalfont
Consider a complete linear differential equation with initial condition $x_0\sim\text{Gamma}(4,9)$ ($4$ is the shape and $9$ is the rate), domain $[t_0,T]=[0,1]$,
\[ a(t,\omega)=\sum_{j=1}^{\infty} \frac{\sqrt{2}}{j^3}\sin\left(tj\pi\right)\xi_j(\omega), \quad  b(t,\omega)=\sum_{i=1}^{\infty} \frac{\sqrt{2}}{i^4+6}\sin\left(ti\pi\right)\eta_i(\omega), \]
where $\xi_1,\xi_2,\ldots,\eta_1,\eta_2,\ldots$ are independent with distribution $\xi_j\sim\text{Uniform}(-\sqrt{3},\sqrt{3})$ and $\eta_j\sim\text{Normal}(0,1)$, $j\geq1$.

The assumptions of Theorem \ref{teor1} hold (see Subsection \ref{comments}). Hence, the density $f_1^N(x,t)$ from (\ref{f1n}) approximates the density function of the stochastic process $x(t,\omega)$ given in (\ref{sol_ale}). 

In Figure \ref{comp2} we have represented $f_1^N(x,t)$ for $N=1$ and $N=2$ at $t=0.4$. 

\begin{figure}[H]
  \begin{center}
    \includegraphics[width=7cm]{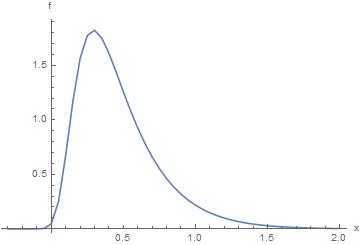}
		\includegraphics[width=7cm]{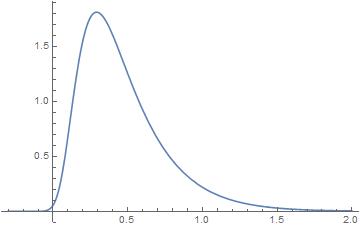}
    \caption{Density $f_1^N(x,t)$ for $N=1$ (left) and $N=2$ (right) at the point $t=0.4$. Example \ref{complete2}.}
		\label{comp2}
    \end{center}
  \end{figure}

\end{example}

\section{Conclusions}

In this paper we have determined approximations for the probability density function of the solution to the randomized non-autonomous complete linear differential equation. This solution is a stochastic process expressed by means of Lebesgue integrals of the data stochastic processes, therefore its probability density function cannot be obtained in an exact manner and approximations for it are required.

The main ideas of this paper can be summarized as follows. Using Karhunen-Loève expansions of the data stochastic processes and truncating them, we obtained a truncation of the solution stochastic process. The Random Variable Transformation technique permitted us to obtain the exact probability density function of the truncation process, which, intuitively, approximates the density function of the solution stochastic process. Theorems \ref{teor1}, \ref{teor2}, \ref{teor3} and \ref{teorb0L} give the conditions under which the approximations constructed via these truncations converge uniformly, while Theorems \ref{teornou1}, \ref{teornou2}, \ref{avo}, \ref{teornou3} and \ref{teornoub0L} give conditions for pointwise convergence. Theorem \ref{teor2}, Theorem \ref{teorb0L}, Theorem \ref{teornou2}, Theorem \ref{avo} and Theorem \ref{teornoub0L} deal with the non-autonomous homogeneous linear differential equation, which give interesting and particular results when there is no source term in the random differential equation.

It is remarkable that, depending on the way the Random Variable Transformation technique is applied (which variable is essentially isolated when computing the inverse of the transformation mapping), we get different but equivalent expressions for the probability density function of the truncation process. This permitted us to have different hypotheses (theorems) under which the approximating density functions converge. Hence, the generality achieved in terms of applications is notable, as our numerical examples illustrated. 

In the numerical experiments we dealt with both Gaussian and non-Gaussian data stochastic processes for which their Karhunen-Loève expansion is available. It was evinced that the convergence to the probability density function of the solution process is achieved quickly.

\section*{Acknowledgements}
This work has been  supported by the Spanish Ministerio de Econom\'{i}a y Competitividad grant MTM2013-41765-P.

\section*{Conflicts of interest} None.

\end{document}